\documentclass[journal]{IEEEtran}
\usepackage{graphicx}
\usepackage{amsfonts}
\usepackage{amsmath}
\usepackage{cite}
\usepackage{color}
\usepackage{amssymb}
\usepackage{enumerate}
\newtheorem{theorem}{\bf Theorem}

\newtheorem{lemma}{\bf Lemma}

\newtheorem{assumption}{\bf Assumption}
\newenvironment{proof}{{\it \bf Proof:}}{\hfill $\blacksquare$\par}
\newtheorem{Remark}{\bf Remark}

\begin{document}

\title{Optimally combined incentive for cooperation among interacting agents in population games}

\author{Shengxian~Wang,~Ming~Cao, \IEEEmembership{Fellow, IEEE}, and~Xiaojie~Chen
\thanks{This research was supported by the National Natural Science Foundation of China (Grant Nos. 62036002 and 61976048). S.W. acknowledges the support from China Scholarship Council (Grant No. 202006070122). (Corresponding authors: Ming Cao and Xiaojie Chen).}
\thanks{S. Wang is with School of Computer and Information, Anhui Normal University, Wuhu 241002, China, also with the School of Mathematical Sciences, University of Electronic Science and Technology of China, Chengdu 611731, China, and also with ENTEG, Faculty of Science and Engineering, University of Groningen, 9747 AG Groningen, (e-mail: shengxian.wang@ahnu.edu.cn).}
\thanks{M. Cao is with ENTEG, Faculty of Science and Engineering, University of Groningen, Groningen 9747 AG, The Netherlands (e-mail: m.cao@rug.nl).}
\thanks{X. Chen is with School of Mathematical Sciences, University of Electronic Science and Technology of China, Chengdu 611731, China (e-mail: xiaojiechen@uestc.edu.cn).}}

\maketitle

\begin{abstract}
Combined prosocial incentives, integrating reward for cooperators and punishment for defectors, are effective tools to promote cooperation among competing agents in population games. Existing research concentrated on how to adjust reward or punishment, as two mutually exclusive  tools, during the evolutionary process to achieve the desired proportion of cooperators in the population, and less attention has been given to exploring a combined incentive-based control policy that can steer the system to the full cooperation state at the lowest cost. In this work we propose a combined incentive scheme in a population of agents whose conflicting interactions are described by the prisoner's dilemma game on complete graphs and regular networks, respectively. By devising an index function for quantifying the implementation cost of the combined incentives, we analytically construct  the optimally combined incentive protocol by using optimal control theory.  By means of theoretical analysis, we identify the mathematical conditions, under which the optimally combined incentive scheme requires the minimal amount of cost.  In addition to numerical calculations, we further perform computer simulations to verify our theoretical results and explore their robustness on different types of network structures.
 \end{abstract}



\IEEEpeerreviewmaketitle

\section{Introduction}
\label{sec:introduction}
\IEEEPARstart{C}{ooperation} is extremely important in many realistic scenarios ranging from bees collectively carrying food~\cite{Seeley: 1988, Dugatkin: 2000} to countries taking measures to prevent the spread of the coronavirus (COVID-19) pandemic~\cite{Moon: 2020, Momtazmanesh: 2020}.
Driven by self-interests, individuals tend to pursue higher personal benefits by choosing antisocial behavior (e.g., defection), which would undermine the emergence of cooperation~\cite{Vincent:CUP,Axelrod:BBNY,Rand:TCS}. Therefore, how to promote the emergence and maintenance of cooperation in a population of
strategically interacting agents is a formidable challenge~\cite{Sunhb2022, Hu: 2021, RamaziTAC2017, Sunpnas2022}. In the last decades, evolutionary game theory has emerged as a powerful mathematical tool to investigate the problem of cooperation in social dilemmas \cite{Smith: 1982,  Gintis: 2009, Santos:2008, RiehlTAC2016, LiNC2020,  Zhang: 2018, MadeoTAC2014, GovaertTAC2020, Chen2023TAC}, such as the classic  prisoner's dilemma game~\cite{Nowak:1992, LiTEC2016, HilbePNAS13}.

Note that in a  social dilemma, without considering additional control measures into the game system,  there may be no evolutionary advantage to cooperation~\cite{Sasaki: 2012}.  To intervene in such an adverse situation,  prosocial incentives, such as reward for behaving prosocially and punishment for behaving antisocially, can be used to sustain cooperation among unrelated and competing agents~\cite{Riehl1: 2018, Vasconcelos: 2013, Paarporn: 2018,  Ntemos: 2021, Zhutac2022, TanTAC2016, MorimotoTAC2015}. 
Recent research has also studied prosocial incentives incorporating reward and punishment~\cite{Chen: 2015, Fang:2019PRSA}.
For example, 
Fang \emph{et al.} studied spatial public goods game with the synergy effect of punishment and reward, and found that cooperation can be more easily promoted by frequently implementing punishment and reward~\cite{Fang:2019PRSA}.  Chen \emph{et al.} proposed an optimal allocation policy of incentives, that is, ``first carrot, then stick" depending on the actual cooperation state in the population, and the policy is  shown, surprisingly,  to foster cooperation~\cite{Chen: 2015}.
It is stressed that in these works, rewards and punishments are applied separately during the process of incentive implementation. 
However, the simultaneous application of these two is usually observed in human society~\cite{Alventosa: 2021, Smith: 1977,  Allen: 1981, Willard: 2020}. For example, in the company regulation, employees are monetary fined for their tardiness and absence, and those with perfect attendance during an entire month get additional bonuses~\cite{Smith: 1977,  Allen: 1981}. In order to maintain the long-term stability and harmony of the society,
punishing the bad (e.g., terrorists, criminals) and rewarding the good (e.g., philanthropist) are two indispensable means for a country to guide social behavior~\cite{Willard: 2020}.
To capture these phenomena regarding the mixing usage of reward and punishment (collectively named ``\emph{combined incentive}"), it is critical to explore the evolutionary dynamics of cooperation driven by the combined incentive from a theoretical perspective, and further determine how much the combined incentive is needed for enhancing cooperation in a population.

Furthermore, providing  incentives inevitably incurs a cost for social institutions, and several studies have investigated a low-cost incentive  policy for a sufficiently good outcome of cooperation~\cite{Duong: 2021}.  For example, Sasaki \emph{et al.}  studied institutional reward and punishment in the public goods game, and found that punishment can guarantee full cooperation at a significantly lower cost, when compared with reward~\cite{Sasaki: 2012}.  Chen \emph{et al.}  studied the optimal incentive allocation scheme above, and this scheme leads to a lower cost of cooperation than the independent usage of reward or  punishment~\cite{Chen: 2015}.  In these works, however, the cost factor is not thought of as an optimization objective; that is, the above-mentioned intervention policy is not the optimal incentive-based control scheme, and  it is solely obtained by performing simulations. Subsequently,  Wang~\emph{et al.} respectively applied the optimal control theory to theoretically derive the protocols of institutional reward and punishment with the minimal cost in a well-mixed ~\cite{Wang:2019} or structured ~\cite{Wang2022JRCS} population.  To date, it is unclear whether there exists the protocol of the optimally combined incentive with the lowest cost for promoting cooperation in a population? Furthermore, if it exists, under what theoretical conditions would the optimally combined incentive have a cost advantage over optimal reward and punishment?

In this paper, we  thus consider the prisoner's dilemma game with the combined incentive in a population, and  formulate
a cost minimization framework to study the evolutionary dynamics of cooperation. To be more specific, our main goal is to determine how much the combined incentive is needed for cooperation in a population.  We use a graph to represent the population structure, and focus on the complete graphs and regular networks in this work, so that we can theoretically obtain the optimally combined incentive protocols. Given a network structure,  we first obtain the dynamical equation for the evolution of cooperation with the combined incentive. Then by analyzing the stability of equilibria for the system equation, we determine the requested  amount of incentives to drive the system to a cooperation state. Subsequently, utilizing optimal control theory~\cite{Evans: 2005, Geering: 2007}, we establish an index function for quantifying the implementation cost of the combined incentive, and further derive
the optimally combined incentive protocols leading to the minimal cumulative cost for the evolution of cooperation. Interestingly, we find that  the optimally combined incentive protocols are identical and time-invariant for these two kinds of network structures above. We further identify the theoretical conditions, under which applying optimally combined incentive is always the most cost-effective scheme. Besides analytical calculations, we perform computer simulations and confirm that our results are valid for a broad class of networks. 
The main contributions of our work are listed as follows:
\begin{itemize}
	\item  We derive the dynamical equations for the evolution of cooperation on complete graphs and regular networks with the combined incentive, respectively.
	\item We obtain the theoretical conditions of the minimally requested amounts of the combined incentives for the expected cooperation state.
	\item The optimally  combined incentive protocol with the minimal cumulative cost is theoretically obtained by using the approach of the Hamilton-Jacobi-Bellman equation~\cite{Evans: 2005, Geering: 2007}.  We find that optimally combined incentive protocols are identical and time-invariant  on both of the two network structures.
	\item  We identify the mathematical condition under which the cumulative cost required by optimally combined incentive is lowest.  We further perform numerical calculations and computer simulations to support our results and explore their robustness for different types of networks.
\end{itemize}
The remainder of this paper is organized as follows. In Section II, we formalize problem formulation. In Section III, we present the obtained optimally combined incentive protocols. In Section IV, we provide a comparative analysis among optimally combined incentives. Then numerical calculations and computer simulations are performed to validate our obtained theoretical results in Section V. Finally, 
Section VI concludes the paper. Some theoretical proofs are provided in the Appendix.


\begin{figure}[!t]
	\begin{center}
		\includegraphics[width=3.5in]{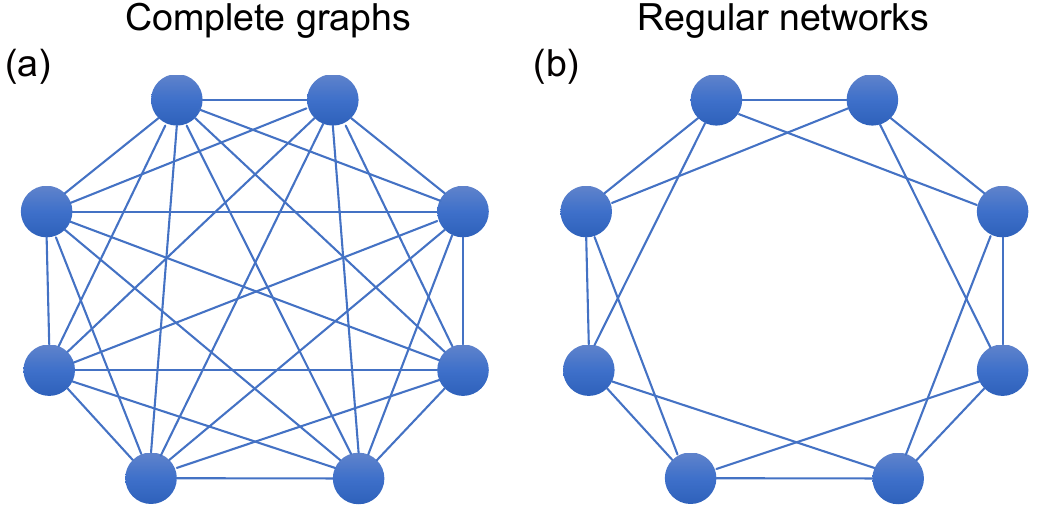}
		\caption{Network representations. Complete graphs depict the topology of a well-mixed population in panel $(a)$, whereas regular networks represent the topology of a structured case in panel $(b)$. Here, complete graphs contain 8 nodes with degree 7, and regular networks include 8 nodes with degree 4.}\label{fig1}
	\end{center}
\end{figure}

\section{Problem formulation}
\subsection{Prisoner's dilemma game with the combined incentive}

We study evolutionary games in  a population of $n\in\mathbb{N_{+}}$ agents whose interaction structure is characterized by a \emph{connected} network. The nodes  correspond to agents, where  all of these nodes are denoted by the set $\mathcal{V}=\{1, \dots, n\}$, and the edges represent who interacts with whom. Each agent $i\in\mathcal{V}$ plays the evolutionary prisoner's dilemma game with its neighbors, who chooses either to cooperate ($C$) or defect ($D$). More specifically, it can choose either to be a cooperator who confers a benefit $b$ to its opponent at a cost $c$ ($0<c<b$)  to itself or to be a defector who provides no help and saves cost. Accordingly, the payoff matrix for the game is
\begin{equation}
	\bordermatrix{
		&C  &D  \cr
		C &b-c & -c\cr
		D &b & 0 }.
	\label{matrix}
\end{equation}\label{eq1}
It can be observed  that  defection for each agent is the dominant choice in this game. However, if both agents choose $C$, then they could  yield a higher  payoff. But when both choose to defect, both get nothing. This inevitably leads to a social dilemma of cooperation.

To tackle this dilemma, we  then introduce a combined incentive policy  involving both reward (positive incentive, $R$) and punishment (negative incentive, $P$) to  support cooperation.
We consider that the incentive-providing institution always utilizes  positive incentive with the preference $p\in[0, 1]$, where each cooperator will be rewarded with an amount $pu_{R}$  received from the central institution when interacting with a neighbor, whereas using negative incentive with the preference $1-p$, where each defector will be fined by an amount  $(1-p)u_{P}$  for each interaction.
Specifically,  if playing an interaction with a $C$-agent, then the payoffs of each cooperator and defector are respectively
$\pi_{C}^C=b-c+pu_{R}$
and
$\pi_{D}^C=b-(1-p)u_{P}$. By contrast, if playing an interaction with a $D$-agent, then the payoffs of each cooperator and  defector are respectively
$\pi_C^D=-c+pu_{R}$
and
$\pi_D^D=-(1-p)u_{P}$.
It is noteworthy that  each $p$-value corresponds to a form of the combined incentive.  Moreover,  $p=1$ and $p=0$ separately represent pure reward and punishment. For convenience, we consider that $u_{P}=u_{R}=u$ in this study.

In this work, we mainly focus on well-mixed and regularly  structured populations, which are represented by  complete graphs and regular networks (see Fig.~\ref{fig1}),  respectively.
\subsubsection{Complete graphs} Using a complete graph of $n$ nodes  to characterize an infinitely large, well-mixed population, the expected payoffs of cooperators and defectors are respectively  given by  \cite{Schuster: 1983, Hofbauer: 1998}
\begin{equation}\label{eq2}
	\begin{aligned}
		f_C=\pi_{C}^Cx+\pi_C^Dy=(b-c+pu_{R})x+(-c+pu_{R})y,
	\end{aligned}
\end{equation}
and
\begin{equation}\label{eq3}
	\begin{aligned}
		f_D=\pi_{D}^Cx+\pi_D^Dy=\big[b-(1-p)u_{P}\big]x-(1-p)u_{P}y.
	\end{aligned}
\end{equation}
Without loss of generality,  $x$ and $y$ respectively denote the  proportions of cooperators and defectors in the whole population, which satisfy $x+y=1$ with  $x, y \in[0,1]$. Notably, both the notations above will be used throughout the paper.

\subsubsection{Regular  networks}

Consider a regular network of $n$ nodes with the  degree $k > 2$ ($k\in\mathbb{N_{+}}$) to capture the pairwise  interactions between neighboring  agents in structured populations. Each agent $i\in\mathcal{V}$ receives a payoff  after playing  with one neighbor, and obtains the accumulated payoff $\Pi_i $ by interacting with its all neighbors.

Based on the evolutionary principle, each  agent has the same chance to revise its strategy by comparing the above-mentioned accumulated payoff with the one obtained by a neighboring agent. Accordingly, at each time step we randomly select two neighboring $i$ and $j$ agents with $i, j \in\mathcal{V}$, and calculate their  total payoffs (i.e., $\Pi_{i}$ and $\Pi_{j}$). Then, agent $i$ may adopt the strategy of the neighboring $j$ with
the probability defined by the Fermi function \cite{Szabo: 1998}
\begin{equation}\label{eq4}
	W_{i\rightarrow j}=\frac{1}{1+e^{-\omega(\Pi_{j}-\Pi_{i})}},
\end{equation}
where $\omega\in[0,1]$ measures the strength of selection.  In the $\omega\rightarrow 0$ limit, called the  weak selection limit,  agent $i$ will stay with its own strategy or  imitate the strategy of neighbor $j$ with equal probabilities.
When $\omega>0$, the more successful neighbor $j$ is, the more likely it is that agent $i$ will adopt the strategy of neighbor $j$.

According to the strategy update rule above, we randomly choose two neighboring agents who adopt different strategies. Consequently,
the accumulated payoff of a $C$-agent who has a selected $D$-neighbor and other $k-1$ neighbors comprised of $l$ cooperators and  $k-l-1$ defectors
is
\begin{equation}\label{eq5}
	\begin{aligned}
		\Pi_C&=\sum_{l=0}^{k-1} \binom{k-1}{l} (x_{C|C})^{l}(x_{D|C})^{k-l-1} \Big[\pi_{C}^C l+\pi_{C}^D(k-l)\Big]\\
		&=(k-1)x_{C|C}b+k(pu-c),
	\end{aligned}
\end{equation}
and similarly, the accumulated payoff of a $D$-agent who has a selected $C$-neighbor and other $k-1$ neighbors comprised of $l$ cooperators and  $k-l-1$ defectors is
\begin{equation}\label{eq6}
	\begin{aligned}
		\Pi_D&=\sum_{l=0}^{k-1} \binom{k-1}{l} (x_{C|D})^{l}(x_{D|D})^{k-l-1} \Big[\pi_{D}^C (l+1)\\
		&\quad +\pi_{D}^D(k-l-1)\Big]\\
		&=\big[(k-1)x_{C|D}+1\big]b-k(1-p)u.
	\end{aligned}
\end{equation}
Here, $x_{i|j}$ denotes the conditional probability to  find an $i$-agent given that the neighboring node is occupied by a $j$-agent, where  $i, j \in \{C, D\}$ and $x_{i|j}\in[0, 1]$.

\subsection{Evolutionary dynamics of cooperation driven by the combined incentive}

Subsequently, we use the replicator equation (on graphs) to characterize the dynamical changes of the fraction of cooperators in a population over time $t$.
\subsubsection{Complete graphs}  In an infinitely large, well-mixed population, the replicator equation with the combined incentive  is described by \cite{Schuster: 1983, Hofbauer: 1998}
\begin{equation}\label{eq7}
	\begin{aligned}
		\frac{\textrm{d}x(t)}{\textrm{d}t}=x(1-x)(f_C-f_D).
	\end{aligned}
\end{equation}

Combining \eqref{eq2}, \eqref{eq3} and \eqref{eq7},  the system equation is rewritten as
\begin{equation}\label{eq8}
	\frac{\textrm{d}x(t)}{\textrm{d}t}=F_{I}(x, u, t)=x(1-x)(u-c).
\end{equation}
\subsubsection{Regular networks} Using the pair approximation approach~\cite{Ohtsuki: 2006, Ohtsuk:2006Nature}, the replicator equation  with the combined incentive  in structured populations  is described by   (see details in Appendix A)
\begin{equation}\label{eq9}
	\frac{\textrm{d}x(t)}{\textrm{d}t}=F_{S}(x, u, t)=\frac{\omega k(k-2)}{2(k-1)}x(1-x)(u-c).
\end{equation}

\begin{assumption}\label{assumption1}
	For any structured network, the population structure, the payoff matrix, and the strength of selection $\omega$ are fixed.
\end{assumption}
Assumption \ref{assumption1}  means that the considered parameters (i.e.,  $n$, $\omega$, $k$, $b$ and $c$) do not change over $t$.
Under Assumption \ref{assumption1}, for system equations obtained in \eqref{eq8} or \eqref{eq9}, we have the following two remarks.
\begin{Remark}\label{Remark1}
	The rewarding preference $p$ is not directly associated with the above two system equation, and it does not affect the evolutionary dynamics of the system.
\end{Remark}

\begin{Remark}\label{Remark2}
	In the absence of incentives (i.e., $u=0$), the system equation \eqref{eq9} is consistent with the one obtained in Ref. \cite{{Ohtsuki: 2006}}.
\end{Remark}


\subsection{Optimal control problem}

Since providing incentives is costly for institutions, the goal of optimization problems for the combined incentive is to minimize the implementation cost of incentives required for supporting cooperation, and find the explicit expressions of optimal incentive protocol $u^\ast$. To this end, we separately formulate an optimal control problem for the combined incentive on complete graphs and regular networks as follows.
\subsubsection{Complete graphs} The optimal control problem  is given as
\begin{eqnarray}\label{eq10}
	&\min\,\,&J_{I}=\int^{t_{f}}_{t_0}G_{I}(x, u, t)\textrm{d}t\nonumber\\
	& s.t.&\quad  \left\{\begin{array}{lc}
		\frac{\textrm{d}x(t)}{\textrm{d}t}=F_{I}(x, u, t),  \\
		x(t_0)=x_{0}, \\
		x(t_{f})=1-\delta,
	\end{array}\right.
\end{eqnarray}
where $G_{I}(x, u, t)=\frac{1}{2}\big[n(n-1) xpu_R+n(n-1)(1-x)(1-p)u_P\big]^{2}=\frac{1}{2}\big\{ n(n-1)u [p x+(1-p)(1-x)] \big\}^2.$
\subsubsection{Regular networks}The control problem  is given as
\begin{eqnarray}\label{eq11}
	&\min\,\,&J_{S}=\int^{t_{f}}_{t_0} G_{S}(x, u, t) \textrm{d}t\nonumber\\
	& s.t.&\quad  \left\{\begin{array}{lc}
		\frac{\textrm{d}x(t)}{\textrm{d}t}=F_{S}(x, u, t),\\
		x(t_0)=x_{0},\\
		x(t_{f})=1-\delta,
	\end{array}\right.
\end{eqnarray}
where $G_{S}(x, u, t)=\frac{1}{2}\big[kn xpu_R+kn (1-x)(1-p)u_P\big]^{2}=\frac{1}{2} \big\{ nku[px+(1-p)(1-x)] \big\}^2$.

For the optimal control problem above, the per capita incentive $u$ is regarded as the control variable, and the cumulative cost $J_{v}  (v\in\{I, S\})$ as the objective function. The cost function $J_{v}$ describes the cumulative cost on average from the initial time $t_0$  to the terminal time $t_{f}$ for the dynamical system $\frac{\textrm{d}x}{\textrm{d}t}=F_{v}(x, u, t)$, which is defined in \eqref{eq8} if $v=I$ and in \eqref{eq9}  if $v=S$.  For convenience,  the initial time is set as $0$ (i.e., $t_0=0$).  Herein, we denote the initial cooperation state (i.e., the initial proportion of cooperators) in the population by $x(t_0)=x_{0}$, and denote the desired or terminal cooperation level (i.e., the desired proportion of cooperators) in the population by $x(t_{f})$.  With the aim of exploring the optimally combined incentive protocol over the time interval $[t_0, t_{f}]$,  we now impose two assumptions for the optimal control problem.
\begin{assumption}\label{assumption2}
	The terminal time  $t_{f}$ is free (i.e., not fixed).
\end{assumption}
Assumption \ref{assumption2} stipulates that the optimal function $J_{v}^{\ast}(x, t)$ is independent of $t$, that is,  $\frac{\partial J_{v}^\ast(x, t)}{\partial t}=0$.

\begin{assumption} \label{assumption3}
	The terminal state $x(t_{f})$ is fixed, and set to be $1-\delta$ (i.e., $x(t_{f})=1-\delta>x_{0}$), where $\delta$ is the parameter that determines the expected cooperation level at $t_{f}$.
\end{assumption}
Assumption \ref{assumption3} means that  the desired cooperation level $x(t_{f})$ can be reached at the terminal time $t_{f}$.   Therefore, the quantity $\min J_{v}$ can be regarded as the objective functional of calculating the optimally combined incentive protocol $u^\ast$ with the lowest  cumulative cost.

\subsection{Monte Carlo simulations}
During a full Monte Carlo step, on average each agent has a chance to update its strategy.  The applied two different update rules are  specified on complete graphs and other networks, respectively. On complete graphs, any agent $i$ has an opportunity to imitate the strategy of another randomly chosen  agent $j$ and calculate their payoff values (i.e.,  $f_i$ and $f_j$). If $f_i<f_j$, then agent $i$ will transfer its strategy with the probability
\begin{equation}\label{eq12}
	Q_{i\rightarrow j}=\frac{f_j-f_i}{M},
\end{equation}
where $M$ ensures the proper normalization and is given by the maximum possible difference between the payoffs of $i$ and $j$ agents~\cite{Santos:2008}. We stress that under this update rule which contains stochastic elements and utilizes microscopic dynamics, the macroscopic replicator equation on complete graphs can be obtained~\cite{Santos:2008}. We use the pairwise-comparison update rule as specified in \eqref{eq4} for the prisoner's dilemma on other types of networks including regular, scale-free~\cite{Barab:1999}, random~\cite{Erd:1959}, and small-world ~\cite{Watts:1998} networks.


	\section{Optimally combined incentive protocol}
	
	\subsection{Condition for the prevalence of cooperation}
	In this subsection, we analyze the existence and (asymptotic) stability of equilibria for the system equation obtained in  \eqref{eq8} or \eqref{eq9}, and the lemma is stated below.

		\begin{lemma}\label{lem1}
			On complete graphs or regular networks, the following statements hold:
			\begin{enumerate}[(1)]
				\item There always exist two equilibria of \eqref{eq8}  or \eqref{eq9}, namely $x^\ast=1$ and $x^\ast=0$;
				\item If $u> c$ and $x_0\in (0, 1)$, the equilibrium point $x^\ast=1$ is (asymptotically)  stable and $x^\ast=0$ is unstable;
				\item If $u< c$ and $x_0\in (0, 1)$, the equilibrium point $x^\ast=1$ is unstable and $x^\ast=0$ is (asymptotically)  stable.
			\end{enumerate}
		\end{lemma}
		
		\noindent \begin{proof}
			Let us first consider complete graphs, and the system equation is described by \eqref{eq8}. Solving $F_{I}(x, u, t)=0$, it is easy to check that Eq.~\eqref{eq8} has two equilibria, namely, $x^\ast=1$ and  $x^\ast=0$.
			Since the derivative of  $F_{I}(x, u,  t)$  with respect to $x$  is
			\begin{equation}\label{eq13}
				\begin{aligned}
					\frac{\textrm{d}F_{I}}{\textrm{d}x}=(1-2x)(u-c),
				\end{aligned}
			\end{equation}
			then at these two equilibria one yields that $\frac{\textrm{d}F_{I}}{\textrm{d}x}\big|_{x^\ast=1}=-(u-c)$ and $\frac{\textrm{d}F_{I}}{\textrm{d}x}\big|_{x^\ast=0}=u-c$.
			When $u>c$  and  $x_0\in (0, 1)$, one gets that $\frac{\textrm{d}F_{I}}{\textrm{d}x}\big|_{x^\ast=1}<0$ and $\frac{\textrm{d}F_{I}}{\textrm{d}x}\big|_{x^\ast=0}>0$ for $\forall \,\, t\geq0$. Thus,  the equilibrium point  $x^\ast=1$ is stable, and $x^\ast=0$ is unstable. On the other hand, when $u<c$  and  $x_0\in (0, 1)$, one checks that $\frac{\textrm{d}F_{I}}{\textrm{d}x}\big|_{x^\ast=1}>0$  and $\frac{\textrm{d}F_{I}}{\textrm{d}x}\big|_{x^\ast=0}<0$ for $\forall \,\, t\geq0$. Therefore, the equilibrium point   $x^\ast=1$ is unstable, and  $x^\ast=0$ is stable.
			
			Furthermore, we study the alternative case of  regular networks, and the system equation is depicted by \eqref{eq9}.
			Solving $F_{S}(x, u, t)=0$,  Eq.~\eqref{eq9} only has two equilibria, namely, $x^\ast=1$ and  $x^\ast=0$.
			Since the derivative of  $F_{S}(x, u,  t)$ with respect to $x$  is
			\begin{equation}\label{eq14}
				\begin{aligned}
					\frac{\textrm{d}F_{S}}{\textrm{d}x}=\frac{\omega k(k-2)(u-c)}{2(k-1)}(1-2x),
				\end{aligned}
			\end{equation}
			then  at these two equilibria one derives that $\frac{\textrm{d}F_{S}}{\textrm{d}x}\big|_{x^\ast=1}=-\frac{\omega k(k-2)(u-c)}{2(k-1)}$ and $\frac{\textrm{d}F_{S}}{\textrm{d}x}\big|_{x^\ast=0}=\frac{\omega k(k-2)(u-c)}{2(k-1)}$.
			When $u>c$ and $x_0\in (0, 1)$,  one gets that $\frac{\textrm{d}F_{S}}{\textrm{d}x}\mid_{x^\ast=1}<0$ and  $\frac{\textrm{d}F_{S}}{\textrm{d}x}\mid_{x^\ast=0}>0$  for $\forall \,\,t\geq0$. Thus, the equilibrium point $x^\ast=1$ is stable, and  $x^\ast=0$ is unstable.
			In contrast, when
			$u<c$ and $x_0\in (0, 1)$, one checks that $\frac{\textrm{d}F_{S}}{\textrm{d}x}\big|_{x^\ast=1}>0$ and $\frac{\textrm{d}F_{S}}{\textrm{d}x}\big|_{x^\ast=0}<0$ for $\forall \,\, t\geq0$. Therefore, the equilibrium point
			$x^\ast=1$ is unstable, and $x^\ast=0$ is stable.
		\end{proof}
		Lemma \ref{lem1} illustrates that the governing system equation with the combined incentive has two equilibrium points  on both complete graphs and regular networks, which are respectively the full defection state and the full cooperation state. For each given network structure, we find that  cooperation is more abundant than defection, if and only if $u>c$.  Nevertheless, providing the combined incentive is costly, and thus under the condition of $u>c$, it is desirable to find an incentive-based control policy, which ensures the establishment of cooperation at a lower cost.  This will be discussed in the next subsection.
		
		

		\subsection{Optimally combined incentive protocol with minimal cost}
		
		Based on the results presented above, we study the consequence of time-dependent combined incentive. We show  the optimally combined incentive protocol by solving the optimal control problem defined in \eqref{eq10}  or \eqref{eq11}. Then, we analyze the amount of cumulative cost induced by the optimally combined incentive for the dynamical system to reach the expected terminal state from the initial state. The related results of complete graphs are presented in Theorem \ref{thm1} and those of regular networks in Theorem \ref{thm2}.
		
		\begin{theorem}\label{thm1}
			On complete graphs,  the optimally combined  incentive protocol for the optimal control problem \eqref{eq10} is
			\begin{equation}\label{eq15}
				\begin{aligned}
					u^{\ast}=2c.
				\end{aligned}
			\end{equation}
			With the optimally combined incentive protocol, the solution of the system~\eqref{eq8} is
			\begin{equation}\label{eq16}
				\begin{aligned}
					x(t)=\frac{1}{1+\frac{1-x_{0}}{x_{0}} e^{-ct}},
				\end{aligned}
			\end{equation}
			and its cumulative cost is

			\begin{equation}\label{eq17}
				\begin{aligned}
					J_{I}^\ast&=\int^{t_{f}}_{t_0} G_{I}(x, u^{\ast}, t) \textrm{d}t\\
					&=\int^{t_{f}}_{t_0}\frac{1}{2}\Big\{ n(n-1)u^\ast \big[p x+(1-p)(1-x)\big] \Big\}^2\textrm{d}t\\
					&=\alpha_{I} p^2+2\beta_{I} p(1-p)+\gamma_{I}(1-p)^2\\
					&=\vartheta_I \Big[(2p-1)^2(x_{0}-1+\delta)+p^2\ln\Big(\frac{1-x_0}{\delta}\Big)\\
					&+(1-p)^2\ln\Big(\frac{1-\delta}{x_0}\Big)\Big],
				\end{aligned}
			\end{equation}
			where $\vartheta_I=2n^2(n-1)^2c$, $\alpha_{I}=\frac{1}{2}\int^{t_{f}}_{t_0}[n(n-1)xu^{\ast}]^{2}\textrm{d}t=\vartheta_I[x_{0}-1+\delta+\ln(\frac{1-x_0}{\delta})],$
			$\beta_{I}=\frac{1}{2}\int^{t_{f}}_{t_0}[n(n-1)u^{\ast}]^{2}x(1-x)\textrm{d}t=\vartheta_I(1-x_0-\delta),$ and
			$\gamma_{I}=\frac{1}{2}\int^{t_{f}}_{t_0}[n(n-1)(1-x)u^{\ast}]^{2}\textrm{d}t=\vartheta_I[x_{0}-1+\delta+\ln(\frac{1-\delta}{x_0})].$
		\end{theorem}
		\begin{proof}
			See Appendix B.
		\end{proof}
		
		\begin{theorem}\label{thm2}
			On regular networks,  the optimally combined  incentive protocol for the optimal control problem \eqref{eq11}  is
			\begin{equation}\label{eq18}
				\begin{aligned}
					u^{\ast}=2c,
				\end{aligned}
			\end{equation}
			With the optimally combined  incentive protocol, the solution of the system~\eqref{eq9} is
			\begin{equation}\label{eq19}
				\begin{aligned}
					x(t)=\frac{1}{1+\frac{1-x_{0}}{x_{0}} e^{-\frac{\omega k(k-2)c}{2(k-1)}t}}.
				\end{aligned}
			\end{equation}
			and its cumulative cost is
			\begin{equation}\label{eq20}
				\begin{aligned}
					J_{S}^\ast&=\int^{t_{f}}_{t_0} G_{S}(x, u^{\ast}, t) \textrm{d}t\\
					&=\int^{t_{f}}_{t_0}\frac{1}{2} \Big\{nku^\ast \big[px+(1-p)(1-x)\big] \Big\}^2\textrm{d}t\\
					&=\alpha_{S} p^2+2\beta_{S} p(1-p)+\gamma_{S}(1-p)^2\\
					&=\vartheta_{S} \Big[(2p-1)^2(x_{0}-1+\delta)+p^2\ln\Big(\frac{1-x_0}{\delta}\Big)\\
					&+(1-p)^2\ln\Big(\frac{1-\delta}{x_0}\Big)\Big],
				\end{aligned}
			\end{equation}
			where $\vartheta_S=\frac{4n^2k(k-1)c}{\omega(k-2)}$,
			$\alpha_S=\frac{1}{2}\int^{t_{f}}_{t_0}(knxu^{\ast})^{2}\textrm{d}t=\vartheta_S[x_{0}-1+\delta+\ln(\frac{1-x_0}{\delta})],$
			$\beta_S=\frac{1}{2}\int^{t_{f}}_{t_0}(nku^{\ast})^{2}x(1-x)\textrm{d}t=\vartheta_S(1-x_0-\delta),$ and
			$\gamma_S=\frac{1}{2}\int^{t_{f}}_{t_0}[kn(1-x)u^{\ast}]^{2}\textrm{d}t=\vartheta_S[x_{0}-1+\delta+\ln(\frac{1-\delta}{x_0})].$
		\end{theorem}
		
		\begin{proof}
			See Appendix C.
		\end{proof}

		Besides, let $p=1$ or $p=0$, and accordingly we
		give two remarks to elaborate on the special case of reward  and  punishment considered in isolation.
		\begin{Remark}\label{Remark3}
			When reward is used by the institution (i.e.,  $p=1$),  the optimal rewarding protocol is  $u^{\ast}_R=2c$, and the amount of cumulative cost is $J^{\ast}_{R, v}=\alpha_v,$ where $v\in\{I, S\}$.
		\end{Remark}
		
		\begin{Remark}\label{Remark4}
			When punishment is used by the institution (i.e.,  $p=0$),  the optimal punishing protocol is  $u^{\ast}_{P}=2c$, and the amount of cumulative cost is $J_{P, v}^\ast=\gamma_v,$ where $v\in\{I, S\}$.
		\end{Remark}
		
		Combining Theorem~\ref{thm1}, Theorem~\ref{thm2}, Remark~\ref{Remark3} and Remark~\ref{Remark4},   one knows that the obtained optimal incentive protocols for cooperation are $u^\ast=2c$, which are time-invariant and identical for both  complete graphs and regular networks.  When the network structure is determined, the optimally combined incentive  protocol can make the dynamical system converge to the full cooperation state (i.e., $x=1$) for all $p\in[0, 1]$.
		However,  it can be observed that the requested cumulative  cost is determined by the rewarding preference $p$,   the initial cooperation level $x_0$, and the difference $\delta$ between the full cooperation state (i.e., $x=1$) and the desired cooperation state (i.e., $x=1-\delta$).  To this end,
		we will provide comparative analysis between optimal reward and punishment, and among optimally combined incentives for all $p\in[0, 1]$ in the next section, respectively.

		\section{Comparison among optimally combined  incentives}

		In this section, we now proceed to discussing the optimally combined  incentives. We aim to determine a $p^\ast$-value such that the corresponding cumulative cost of the combined incentive $J_{v}^\ast(p^\ast)$ is lowest for $\forall p\in[0, 1]$, that is,  $\mathop{\min}\limits_{p\in [0,1]}J_{v}^\ast(p)=J_{v}^\ast(p^\ast)$ in three scenarios regarding the relation of $x_{0}$ and $\delta$, which are respectively $x_{0}=\delta$, $x_{0}>\delta$, and $x_{0}<\delta$.
		Before comparison,  we first determine the specific expression of the range $\mathcal{Z}=\big\{(x_{0}, \delta)\big| x_{0}+\delta<1,  0< x_{0}<1,  0<\delta<1\big\}$ for each given scenario.
		
		
		\begin{Remark}\label{Remark5}
			The expression of the range $\mathcal{Z}$ is:
			\begin{enumerate}[(1)]
				\item for $x_{0}=\delta$,  $\mathcal{Z}=\big\{(x_{0},  \delta)\big| x_{0}=\delta,\,  0<x_{0}< 0.5\big\}$;
				\item for $x_{0}>\delta$, $\mathcal{Z}=\big\{(x_{0}, \delta)\big| 0<\delta <0.5,  \, \delta<x_{0} < 1-\delta\big\}$;
				\item for  $x_{0}<\delta$,  $\mathcal{Z}=\big\{(x_{0}, \delta)\big| 0<x_{0} <0.5, \, x_{0} <\delta <1-x_{0}\big\}$.
			\end{enumerate}
		\end{Remark}


		For notational simplicity,  from \eqref{eq17} and \eqref{eq20} we can define the cumulative cost of the optimally combined incentive for both distinct network structures as the function
		\begin{equation}\label{eq21}
			\begin{aligned}
				J^*_v(p)&=\alpha_{v} p^2+2\beta_{v}p(1-p)+\gamma_{v}(1-p)^2\\
				&=\vartheta_{v} \Big[(2p-1)^2(x_{0}-1+\delta)+p^2\ln\Big(\frac{1-x_0}{\delta}\Big)\\
				&+(1-p)^2\ln\Big(\frac{1-\delta}{x_0}\Big)\Big],
			\end{aligned}
		\end{equation}
		where $v\in\{I, S\}$, $\alpha_v=\vartheta_v[x_{0}-1+\delta+\ln(\frac{1-x_0}{\delta})]$,
		$\beta_v=\vartheta_v(1-x_0-\delta)$, and
		$\gamma_v=\vartheta_v[x_{0}-1+\delta+\ln(\frac{1-\delta}{x_0})]$.  Therefore, taking the derivative of \eqref{eq21} with respect to $p$, we have
		\begin{equation}\label{eq22}
			\begin{aligned}
				\frac{\textrm{d}J^*_{v}(p)}{\textrm{d}p}=2\big[(\alpha_{v}-2\beta_{v}+\gamma_{v})p+\beta_{v}-\gamma_{v}\big].
			\end{aligned}
		\end{equation}
		Solving $\frac{\textrm{d}J^*_{v}}{\textrm{d}p}=0$, we have
		\begin{equation}\label{eq23}
			\begin{aligned}
				p^\ast&=\frac{\gamma_v-\beta_v}{\alpha_{v}-2\beta_{v}+\gamma_{v}}=\frac{x_0-1+\delta+\ln(\frac{1-x_0}{\delta})}{4(x_0-1+\delta)+\ln\big[\frac{(1-x_0)(1-\delta)}{x_0 \delta}\big]},
			\end{aligned}
		\end{equation}
		when $\alpha_{v}-2\beta_{v}+\gamma_{v}\neq0$.
		Furthermore, we define the cumulative cost difference between $J^\ast_{P, v}$ and $J^\ast_{R, v}$  as the function
		\begin{equation}\label{eq24}
			\begin{aligned}
				\mathcal{G}(x_0, \delta)=J^*_v(1)-J^*_v(0)=\alpha_v-\gamma_v=\vartheta_v\ln\Big[\frac{x_0(1-x_0)}{\delta(1-\delta)}\Big],
			\end{aligned}
		\end{equation}
		and further define the function $ \Theta_v(x_0, \delta)$ as
		\begin{equation}\label{eq25}
			\begin{aligned}
				\Theta_v(x_0, \delta)=x_0(1-x_0)-\delta(1-\delta)=(x_0-\delta)(1-x_0-\delta).
			\end{aligned}
		\end{equation}

		
		
		\subsection{$x_0=\delta$}
		When the initial cooperation level is equal to the difference between the full cooperation  and desired cooperation states (i.e., $x_0=\delta$),
		we have the following conclusion.
		
		\begin{lemma}\label{lem2}
			The inequality  $2(1-2x_{0})<\ln(\frac{1-x_{0}}{x_0})$ always holds, i.e., $\beta_v<\gamma_v$.
		\end{lemma}
		\begin{proof}
			According to Remark \ref{Remark5}, one  can obviously  check $x_{0}\in(0, 0.5)$ since $x_0=\delta$.
			Define the function $g_v(x_{0})$ as
			\begin{equation}\label{eq26}
				\begin{aligned}
					g_{v}(x_{0})=\beta_v-\gamma_v=\vartheta_v\Big[2(1-2x_{0})-\ln\Big(\frac{1-x_{0}}{x_0}\Big)\Big].
				\end{aligned}
			\end{equation}
			Its derivative with respect to  $x_{0}$ is
			$
			\frac{\textrm{d}g_{v}(x_{0})}{\textrm{d}x_{0}}=\vartheta_v \frac{(2x_{0}-1)^2}{x_0(1-x_0)}\textgreater0
			$
			over $(0, 0.5)$, indicating that $g_v(x_{0})$ is monotonically increasing over  $(0, 0.5)$. Furthermore, one can check that
			\begin{equation}\label{eq27}
				\begin{aligned}
					g_{v}(x_{0})<\mathop{\sup}_{x_{0}\in(0, 0.5)}g_{v}(x_{0})=g_{v}(x_{0})\big|_{x_{0}=0.5}=0.
				\end{aligned}
			\end{equation} Thus, $\beta_v<\gamma_v$ holds, which is equivalent to $2(1-2x_{0})<\ln(\frac{1-x_{0}}{x_0})$.
		\end{proof}
		
		\subsubsection{The independent usage of optimal reward or punishment}
		We firstly consider the special case where reward (i.e., $p=1$) and punishment (i.e., $p=0$) are independently used, and discuss both of them.
		\begin{lemma}\label{lem3}
			If $p=1$ or $p=0$,  both cumulative costs for the optimal rewarding and punishing incentives are identical, that is,  $J_{R, v}^\ast=J_{P, v}^\ast$ (i.e., $\alpha_v=\gamma_v$).
		\end{lemma}
		\begin{proof}
			Since $x_0=\delta$,  it follows from \eqref{eq24} and \eqref{eq25} that $ \Theta_v(x_0, \delta)=0$, or equivalently,  $\mathcal{G}(x_0, \delta)=0 $. Therefore, $J^*_v(1)=J^*_v(0)$  (i.e., $J_{R, v}^\ast=J_{P, v}^\ast$ and $\alpha_v=\gamma_v$).
		\end{proof}
		
		\subsubsection{Optimally combined incentive}
		Then, we study the optimally combined incentive with minimal cumulative cost for all $p \in[0, 1]$. Now the theorem is stated below.
		\begin{theorem}\label{thm3}
			If $p^*=0.5$, the cumulative cost of the combined incentive is minimized for $\forall  p\in[0,1]$, that is, $ \min\limits_{p\in [0,1]}J_{v}^\ast(p)=J_{v}^\ast(p^*)$.
		\end{theorem}
		\begin{proof} Since $\alpha_v=\gamma_v$ in Lemma  \ref{lem3}, the derivative of the function $J_v^\ast(p)$  with respect to $p$ defined in \eqref{eq22} is
			$
			\frac{\textrm{d}J^*_{v}}{\textrm{d}p}=2(\gamma_v-\beta_v)(2p-1).
			$
			Solving $\frac{\textrm{d}J^*_{v}(p)}{\textrm{d}p}=0$, we have $p^*=0.5$.
			Since $\beta_v<\gamma_v$ in Lemma \ref{lem2},  then the minimum value of $J_v^\ast(p)$ is taken at $p^*=0.5$ over the internal $[0, 1]$. That is, for $\forall  p\in[0, 1]$, $J_v^\ast(p^*) \leq J_v^\ast(p)$ and $\mathop{\min}\limits_{p\in [0,1]}J_{v}^\ast(p)=J_{v}^\ast(p^*)$.
		\end{proof}
		
		

		By summarizing Lemma  \ref{lem3} and Theorem \ref{thm3} presented above, one can conclude that under the condition of $x_0=\delta$,  when optimal reward and punishment are used  independently, applying optimal rewarding incentive ($p=1$) always induces the same cumulative cost as the optimal punishing incentive ($p=0$), which is consistent with the obtained result in Ref. \cite{Wang2022JRCS}. In contrast, when the optimally combined incentive is considered for any $p\in[0, 1]$, we  can always find a $p^*$ such that the cumulative cost  $J_v^\ast(p)$ of the  optimally combined incentive is lowest over $[0, 1]$. In addition, these theoretical results remain valid for both complete graphs and regular networks.

		\subsection{$x_0>\delta$}
		In the case where the initial cooperation level exceeds the difference between the full cooperation and desired cooperation states (i.e., $x_0>\delta$), we start with studying the relation between $\alpha_v$ and $\beta_v$.
		\begin{lemma}\label{lem4}
			The inequality  $2(1-x_{0}-\delta)<\ln(\frac{1-x_{0}}{\delta})$ holds, that is,  $\alpha_v>\beta_v$.
		\end{lemma}
		\begin{proof}
			According to Remark \ref{Remark5}, one  can obtain that $\delta \in(0,0.5)$ and  $\delta<1-x_{0}<1-\delta$ since $x_0>\delta$.
			To illustrate that  $\alpha_v>\beta_v$ holds, we define the function $ \mathcal{F}_v(x_{0}, \delta)$ as
			\begin{equation}\label{eq28}
				\begin{aligned}
					\mathcal{F}_v(x_{0}, \delta)&=\alpha_v-\beta_v\\
					&=\vartheta_v\Big[\ln\Big(\frac{1-x_{0}}{\delta}\Big)-2(1-x_{0}-\delta)\Big]\\
					&=\vartheta_v\Big\{\big[\ln(1-x_{0})-2(1-x_{0})\big]-(\ln\delta-2\delta)\Big\}.
				\end{aligned}
			\end{equation}
			Intuitively, to prove the function $\mathcal{F}_v(x_{0}, \delta)>0$, we define the function $\Delta_v(x)$ as
			\begin{equation}\label{eq29}
				\begin{aligned}
					\Delta_v(x)=\ln(x)-2x,
				\end{aligned}
			\end{equation}
			whose derivative with respect to $x$  satisfies $\frac{\textrm{d}\Delta_{v}(x)}{\textrm{d}x}=\frac{1-2x}{x}$.
			From this, one can see that  the function $\Delta_v(x)$ is monotonically increasing for $\forall x\in(0, 0.5)$, while monotonically decreasing for $\forall x\in(0.5, 1)$.
			
			Besides,  we define the function $\mathcal{T}_v(\delta)$ as
			\begin{equation}\label{eq30}
				\begin{aligned}
					\mathcal{T}_v(\delta)=\Delta_v(\delta)-\Delta_v(1-\delta)=\ln \delta-\ln(1- \delta)+2-4\delta,
				\end{aligned}
			\end{equation}
			and its derivative  with respect to  $\delta$ is
			$
			\frac{\textrm{d}\mathcal{T}_v(\delta)}{\textrm{d}\delta}=\frac{(2\delta-1)^2}{\delta(1-\delta)}\geq0
			$
			for any $\delta\in(0, 0.5)$.  Therefore, the function $ \mathcal{T}_v(\delta)$ is monotonically increasing over the interval $(0, 0.5)$, and one can get
			\begin{equation}\label{eq31}
				\begin{aligned}
					\mathcal{T}_v(\delta)<\mathop{\sup}_{\delta\in(0, 0.5)} \mathcal{T}_v(\delta)=\mathcal{T}(\delta)\big|_{\delta=0.5}=0.
				\end{aligned}
			\end{equation}
			Thus,   it follows from (\ref{eq30}) that $\Delta_v(\delta)<\Delta_v(1-\delta)$, that is,  $\ln \delta-2\delta<\ln(1-\delta)-2(1-\delta)$.  Due to the fact that  $\delta<1-x_{0}<1-\delta$ and $\Delta_v(x)$ is monotonically increasing in $(0, 0.5)$ and  monotonically decreasing in $(0.5, 1)$, one can get that $\mathop{\max}\big\{\Delta_v(\delta), \Delta_v(1-\delta)\big\}<\Delta_v(1-x_{0})$.  From \eqref{eq28}, one  further yields that  $\mathcal{F}_v(x_{0}, \delta)=\vartheta_v[\Delta_v(1-x_{0})-\Delta_v(\delta)]>0$ and $\alpha_v>\beta_v$.
		\end{proof}
		
		\subsubsection{The independent usage of optimal reward or punishment} We compare optimal rewarding incentive with optimal punishing incentive, and the related result is presented in Lemma \ref{lem5}.
		\begin{lemma}\label{lem5}
			If $p=0$,  the requested cumulative cost of optimal punishing incentive is lower than that of  optimal rewarding incentive, that is,  $J_{R, v}^\ast>J_{P, v}^\ast$ (i.e., $\alpha_v>\gamma_v$).
		\end{lemma}
		\begin{proof}
			From \eqref{eq24} and \eqref{eq25}, one observes that  the sign of $ \Theta_v(x_0, \delta)$  depends on the difference $x_0-\delta$.  Since $x_0>\delta$, we have $ \Theta_v(x_0, \delta)>0$, or equivalently,  $\mathcal{G}(x_0, \delta)>0$. Therefore, $J^*_v(1)>J^*_v(0)$ (i.e., $J_{R, v}^\ast>J_{P, v}^\ast$ and $\alpha_v>\gamma_v$).
		\end{proof}
		
		\subsubsection{Optimally combined incentive} The optimally combined incentive with minimal cost for $\forall p \in[0, 1]$ is analyzed under the conditions of $2(1-x_{0}-\delta)<\ln(\frac{1-\delta}{x_{0}})$ and $2(1-x_{0}-\delta)\geq\ln(\frac{1-\delta}{x_{0}})$ in the following theorem.
		\begin{theorem} \label{thm4}
			The following statements hold:
			\begin{enumerate}[(1)]
				\item  When $2(1-x_{0}-\delta)<\ln(\frac{1-\delta}{x_{0}})$ (i.e.,  $\beta_v<\gamma_v$),  if $p^*=\frac{\gamma_v-\beta_v}{\alpha_{v}-2\beta_{v}+\gamma_{v}}=\frac{x_0-1+\delta+\ln(\frac{1-x_0}{\delta})}{4(x_0-1+\delta)+\ln[\frac{(1-x_0)(1-\delta)}{x_0 \delta}]}$, then  the requested cumulative cost $J_v^\ast(p)$ of the optimally combined incentive is minimized for $\forall  p\in[0, 1]$,  that is, $\min\limits_{p\in [0, 1]}J_v^\ast(p)=J_v^\ast(p^*)$; 
				\item When $2(1-x_{0}-\delta)\geq\ln(\frac{1-\delta}{x_{0}})$  (i.e.,  $\beta_v\geq\gamma_v$), if $p=0$,  then  the requested cumulative cost $J_v^\ast(p)$ of the optimal punishment is minimized for $\forall p\in[0, 1]$, that is, $\min\limits_{p\in [0, 1]}J_v^\ast(p)=J_v^\ast(0)$.
			\end{enumerate}
		\end{theorem}
		\begin{proof}
			We first consider the case of $2(1-x_{0}-\delta)<\ln(\frac{1-\delta}{x_{0}})$ (i.e.,  $\beta_v<\gamma_v$).  Due to the fact that  $\alpha_v>\beta_v$ in Lemma \ref{lem4}, $\alpha_v>\gamma_v$ in Lemma \ref{lem5}, and $\beta_v<\gamma_v$, one can get that $\alpha_v>\gamma_v>\beta_v$, $\alpha_v-2 \beta_v+\gamma_v>0$, and $\frac{\gamma_v-\beta_v}{\alpha_v-2\beta_v+\gamma_v}\in[0,1]$. Since $\alpha_v-2\beta_v+\gamma_v>0$, it follows from \eqref{eq22} and \eqref{eq23} that  $p^\ast=\frac{\gamma_v-\beta_v}{\alpha_v-2\beta_v+\gamma_v}$ is the minimum point of the function $J_v^\ast(p)$ over the internal $[0, 1]$. Therefore, $\min\limits_{p\in [0, 1]}J_v^\ast(p)=J_v^\ast(p^\ast)$ holds.
			
			Next, we consider another case of $2(1-x_{0}-\delta)\geq\ln(\frac{1-\delta}{x_{0}})$  (i.e.,  $\beta_v\geq\gamma_v$). From \eqref{eq22}, the derivative of $J^*_{v}(p)$ with respect to $p$ can be rewritten as
			$
			\frac{\textrm{d}J^*_{v}(p)}{\textrm{d}p}=2[(\alpha_{v}-2\beta_{v}+\gamma_{v})p+\beta_{v}-\gamma_{v}]\\
			=2[(\alpha_{v}-\beta_{v})p+(\beta_{v}-\gamma_{v})(1-p)].
			$
			Since $\alpha_v>\beta_v$ in Lemma \ref{lem4}  and $\beta_v\geq\gamma_v$,   one gets $\frac{\textrm{d}J^*_{v}}{\textrm{d}p}\geq 0$ for $\forall p\in[0, 1]$.
			Therefore, the function  $J^*_{v}(p)$ is monotonically increasing over the internal $[0, 1]$ which means that $\min\limits_{p\in [0, 1]}J_v^\ast(p)=J_v^\ast(0)$.
		\end{proof}
		
		From Lemma \ref{lem5} and Theorem \ref{thm4}, one knows that under the condition of $x_0>\delta$, when optimal reward and punishment are used  independently, applying optimal punishing incentive ($p=0$)  always induces a lower cumulative cost than optimal rewarding incentive ($p=1$) as obtained in Ref. \cite{Wang2022JRCS}.  Then, we consider the combined incentive. We find that when $2(1-x_{0}-\delta)\geq \ln(\frac{1-\delta}{x_{0}})$, applying optimal punishing incentive still leads to a lower cumulative cost. Whereas when $2(1-x_{0}-\delta)<\ln(\frac{1-\delta}{x_{0}})$,  the cumulative cost of the optimally combined incentive is lowest at $p^*$ for any $p\in[0, 1]$ in Theorem  \ref{thm4}.  In addition, these theoretical results remain valid for both  complete  graphs and regular networks.

		\subsection{$x_0<\delta$}
		Lastly, we consider the case where the initial cooperation level is lower than the difference between the full cooperation state and the desired cooperation state (i.e., $x_0<\delta$). In this case, we first carry out the analysis about the relation between $\beta_v$ and $\gamma_v$.
		\\
		\begin{lemma}\label{lem6}
			The inequality  $2(1-x_{0}-\delta)<\ln(\frac{1-\delta}{x_{0}})$ holds,  that is, $\beta_v<\gamma_v$.
		\end{lemma}
		\begin{proof}
			From Remark \ref{Remark5}, one can get that  $x_{0} \in(0,0.5)$ and  $x_{0}<1-\delta<1-x_{0}$ since $x_0<\delta$. To illustrate that  $\beta_v<\gamma_v$ holds, we first define the function $\Lambda_v(x_{0}, \delta)$ as
			\begin{equation}\label{eq32}
				\begin{aligned}
					\Lambda_v(x_{0}, \delta)&=\beta_v-\gamma_v=\vartheta_v\Big[2(1-x_{0}-\delta)-\ln\Big(\frac{1-\delta}{x_{0}}\Big)\Big]\\
					&=\vartheta_v\Big\{\big[2(1-\delta)-\ln(1-\delta)\big]-\big[2x_{0}-\ln (x_{0})\big]\Big\}.
				\end{aligned}
			\end{equation}
			
			Furthermore, to prove the function  $\Lambda_v(x_{0}, \delta)<0$, we define the function  $\Psi_v(x)$ as
			\begin{equation}\label{eq33}
				\begin{aligned}
					\Psi_v(x)=2x-\ln(x),
				\end{aligned}
			\end{equation}
			and its derivative  with respect to $x$ is
			$
			\frac{\textrm{d}\Psi_{v}(x)}{\textrm{d}x}=\frac{2x-1}{x}.
			$
			From this, one can see that  the function $\Psi_v(x)$ is monotonically decreasing over the interval $(0, 0.5)$ and monotonically increasing over the interval $(0.5,1)$.
			
			In addition, we define the function $\Gamma_v(x_0)$ as
			\begin{equation}\label{eq34}
				\begin{aligned}
					\Gamma_v(x_0)&=\Psi(1-x_0)-\Psi(x_0)\\
					&=2(1-2x_0)-\ln(1- x_0)+\ln(x_0).
				\end{aligned}
			\end{equation}
			Therefore, taking the derivative of \eqref{eq34} with respect to $x_0$, we have
			$
			\frac{\textrm{d}\Gamma_v(x_0)}{\textrm{d}x_0}=\frac{(2x_0-1)^2}{x_0(1-x_0)}\geq0
			$
			for any $x_0\in(0, 0.5)$, indicating that the function $\Gamma_v(x_0)$ is monotonically increasing over $(0, 0.5)$. Thus, one can check that
			\begin{equation}\label{eq35}
				\begin{aligned}
					\Gamma_v(x_0)<\mathop{\sup}_{x_0\in(0, 0.5)}\Gamma_v(x_0)=\Gamma_v(x_0)\big|_{x_0=0.5}=0,
				\end{aligned}
			\end{equation}
			and thus $\Psi_v(1-x_{0})<\Psi_v(x_{0})$ holds, that is,  $2(1-x_{0})-\ln(1-x_{0})<2x_{0}-\ln (x_{0})$.  Due to the fact that $x_{0}<1-\delta<1-x_{0}$ and $\Psi_v(x)$ is monotonically decreasing in $(0, 0.5)$, while monotonically increasing in $(0.5, 1)$, we have $\Psi_v(1-\delta)< \mathop{\min}\{\Psi_v(1-x_{0}),  \Psi_v(x_{0})\}$. This  implies $\Psi_v(1-\delta)<\Psi_v(x_{0})$, or equivalently,  $2(1-\delta)-\ln(1-\delta)<2x_{0}-\ln(x_{0})$.
			From \eqref{eq32}, one further derives that  $\Lambda_v(x_{0}, \delta)=\vartheta_v[\Psi_v(1-\delta)-\Psi_v(x_{0})]<0$, and $\beta_v<\gamma_v$.
		\end{proof}
		
		\subsubsection{The independent usage of optimal reward or punishment}
		When reward and punishment are independently used, we have the following results.
		\begin{lemma}\label{lem7}
			If $p=1$,  the requested cumulative cost of optimal rewarding  incentive is lower than that of  optimal punishing incentive, that is,  $J_{R, v}^\ast<J_{P, v}^\ast$ (i.e., $\alpha_v<\gamma_v$).
		\end{lemma}
		\begin{proof}
			Since $x_0<\delta$, it follows from \eqref{eq24} and \eqref{eq25} that $ \Theta_v(x_0, \delta)<0$, or equivalently,   $\mathcal{G}(x_0, \delta)>0$. Consequently, $J^*_v(1)<J^*_v(0)$ (i.e., $J_{R, v}^\ast<J_{P, v}^\ast$).
		\end{proof}
		
		\subsubsection{Optimally combined incentive}
		Then, we respectively analyze the optimally combined incentive with minimal cumulative cost for all $p \in[0, 1]$ corresponding to the two cases of $2(1-x_{0}-\delta)<\ln(\frac{1-x_{0}}{\delta})$ and $2(1-x_{0}-\delta)\geq\ln(\frac{1-x_{0}}{\delta})$, shown in the next theorem.
		\begin{theorem} \label{thm5}
			The following statements hold:
			\begin{enumerate}[(1)]
				\item When $2(1-x_{0}-\delta)<\ln(\frac{1-x_{0}}{\delta})$ (i.e., $\beta_v<\alpha_v$), if $p^*=\frac{\gamma_v-\beta_v}{\alpha_{v}-2\beta_{v}+\gamma_{v}}=\frac{x_0-1+\delta+\ln[\frac{1-x_0}{\delta}]}{4(x_0-1+\delta)+\ln[\frac{(1-x_0)(1-\delta)}{x_0 \delta}]}$, the cumulative cost of the optimally combined incentive is minimized for $\forall  p\in[0, 1]$, that is,  $\min\limits_{p\in [0, 1]}J_v^\ast(p)=J_v^\ast(p^\ast).$
				\item When $2(1-x_{0}-\delta)\geq\ln(\frac{1-x_{0}}{\delta})$ (i.e., $\beta_v\geq\alpha_v$), if $p=1$, then the cumulative cost of the optimally combined incentive is minimized for  $\forall p\in[0, 1]$, that is, $\min\limits_{p\in [0, 1]}J_v^\ast(p)=J_v^\ast(1).$
			\end{enumerate}
		\end{theorem}
		\begin{proof}
			Firstly, we prove the first statement. Due to the fact that   $\alpha_v<\gamma_v$ in Lemma \ref{lem7},  $\beta_v < \gamma_v$ in Lemma \ref{lem6}  and $\beta_v<\alpha_v$, we have $\beta_v<\alpha_v<\gamma_v$, $2 \beta_v-\alpha_v-\gamma_v<0$, and $p^\ast=\frac{\gamma_v-\beta_v}{\alpha_{v}-2\beta_{v}+\gamma_{v}}\in(0,1)$. Since $2 \beta_v-\alpha_v-\gamma_v<0$, it follows from \eqref{eq21} that  the function $J_v^\ast(p)$ is monotonically decreasing over $[0, p^*)$ and monotonically increasing over $[p^*,1]$. Therefore, $p^*$ is the minimum point of the function $J_v^\ast(p)$ over the internal $[0, 1]$, and
			$\min\limits_{p\in [0, 1]}J_v^\ast(p)=J_v^\ast(p^\ast).$
			
			Then, we prove the second statement. From \eqref{eq22}, the derivative of $J^*_{v}(p)$ with respect to $p$ can be rewritten as
			$
			\frac{\textrm{d}J^*_{v}(p)}{\textrm{d}p}=2[(\alpha_{v}-2\beta_{v}+\gamma_{v})p+\beta_{v}-\gamma_{v}]
			=2[(\alpha_{v}-\beta_{v})p+(\beta_{v}-\gamma_{v})(1-p)].
			$
			Since $\alpha_{v}\leq \beta_{v}$ and $\beta_{v}<\gamma_{v}$ in Lemma \ref{lem6}, one can yield that $\frac{\textrm{d}J^*_{v}}{\textrm{d}p}<0$ for $\forall p\in[0, 1]$, and thus  $J^*_{v}(p)$ is monotonically decreasing over $[0, 1]$. Therefore, $\min\limits_{p\in [0, 1]}J_v^\ast(p)=J_v^\ast(1)$ holds.
		\end{proof}
		
		From Lemma \ref{lem7} and  Theorem \ref{thm5}, one can summarize that under the condition of $x_0<\delta$,  when optimal reward and punishment are used  independently, applying optimal rewarding incentive ($p=1$)  always induces a lower cumulative cost than optimal punishing incentive ($p=0$) as obtained in Ref. \cite{Wang2022JRCS}. Then we consider the combined incentive. We find that when $2(1-x_{0}-\delta)\geq\ln(\frac{1-x_{0}}{\delta})$, applying optimal rewarding incentive still leads to a lower cumulative cost, while when $2(1-x_{0}-\delta)<\ln(\frac{1-\delta}{x_{0}})$,  the cumulative cost for the optimally combined incentive is lowest at $p^*$ for any $p\in[0, 1]$ in Theorem  \ref{thm5}.  In addition, these theoretical results remain valid for both  complete  graphs  and regular networks.

		\section{Numerical Results}
		In this section, we provide numerical calculations and Monte Carlo simulations to verify the theoretical results of the optimally combined incentives on different networks obtained in Theorems \ref{thm3}-\ref{thm5}. For any given network structure, we show the amount of cumulative cost $J^*_v(p)$ as a function of the rewarding preference $p$ as depicted in Figs.~\ref{fig2}-\ref{fig6}. Notably, if $p=0$, then the cumulative cost produced by $u^*$ is equal to that produced by $u^*_P$; if $p=1$, then the cumulative cost produced by $u^*$ is equal to that produced by $u^*_R$.

		\begin{figure*}[!t]
			\begin{center}
				\includegraphics[width=7in]{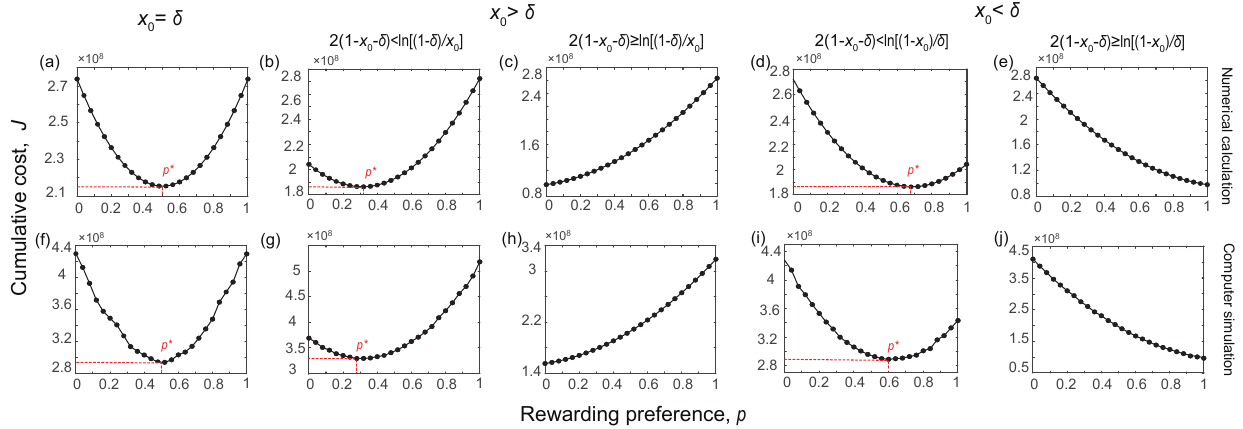}
				\caption{The required amount of cumulative cost to reach the expected proportion of cooperators in the population for the optimally combined incentive in dependence of the rewarding preference $p$ on complete graphs.  First column represents the results under the condition of $x_{0}=\delta$; second and third columns represent the results under the condition of $x_{0}>\delta$, where second column satisfies $2(1-x_{0}-\delta)<\ln(\frac{1-\delta}{x_{0}})$ and other cases are in third column; forth and last columns represent the results under
					the condition of $x_{0}<\delta$, where forth column satisfies $2(1-x_{0}-\delta)<\ln(\frac{1-x_{0}}{\delta})$ and other cases are in last column.  Besides, top row  represents the results derived from numerical calculations based on \eqref{eq8}; bottom row represents the results obtained from Monte carlo simulations by averaging over 200 independent simulation runs. The values of $\delta$  and $x_{0}$ are $\delta = x_{0} =0.1$ (First column); $\delta = 0.1$ and $x_{0} = 0.15$ (Second column); $\delta = 0.1$ and $x_{0}  = 0.3$ (Third column); $\delta = 0.15$ and $x_{0}  = 0.1$ (Forth column); and $\delta = 0.3$ and $x_{0} = 0.1$ (Last column). Parameters: $n=100$, $b=2$,  and  $c=1$. }\label{fig2}
			\end{center}
		\end{figure*}

		\subsection{Complete graphs}
		In order to validate our analytical results on complete graphs, we will present our numerical results from the following three different cases, where $x_{0}=\delta$, $x_{0}>\delta$, and $x_{0}<\delta$, as shown in Fig.~\ref{fig2}.

		\begin{figure*}[!t]
			\begin{center}
				\includegraphics[width=7in]{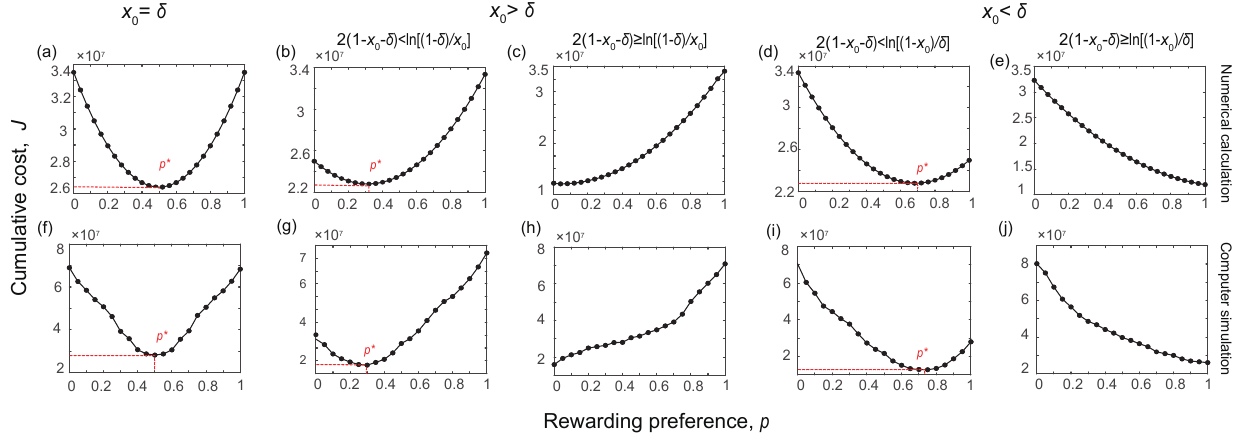}
				\caption{The required amount of cumulative cost to reach the expected proportion of cooperators in the population for the optimally combined incentive in dependence of the rewarding preference $p$ on regular networks. Top row represents the results derived from numerical calculations based on \eqref{eq9}; bottom row represents the results obtained from Monte carlo simulations by averaging over 200 independent simulation runs.  Here, regular networks are  generated by using a two-dimensional square lattice of size $n=l\times l$. Parameters: $l=10$ and $k=4$. Other parameters and descriptions are identical to Fig.~\ref{fig2}.}\label{fig3}
			\end{center}
		\end{figure*}

		\subsubsection{$x_{0}=\delta$}
		From the first column of Fig.~\ref{fig2}, we observe that
		the amount of cumulative cost $J^*_v(p)$ monotonically decreases as the rewarding preference $p$ increases over $[0, p^*)$, while monotonically increases over $[p^*,1]$. Besides, the cumulative cost value $J^*_v(p)$ is lowest at $p^\ast=0.5$ over an internal $[0, 1]$. Moreover, $J_v^*(p)$ is always below the values of $J_v^*(0)$ and $J_v^*(1)$ over an internal $[0,1]$, indicating that the optimally combined incentive is a lower cost scheme compared with optimal reward or punishment. In particular, we can see that the $J_v^*(0)$ value is the same as the $J_v^*(1)$ value, which implies that the cumulative cost of optimal reward is the same cumulative cost as  optimal punishment when the both are used independently. Thus, we verify the theoretical results on complete graphs obtained in Theorem \ref{thm3} and Lemma \ref{lem3}.

		\begin{figure}[!t]
			\begin{center}
				\includegraphics[width=3.5in]{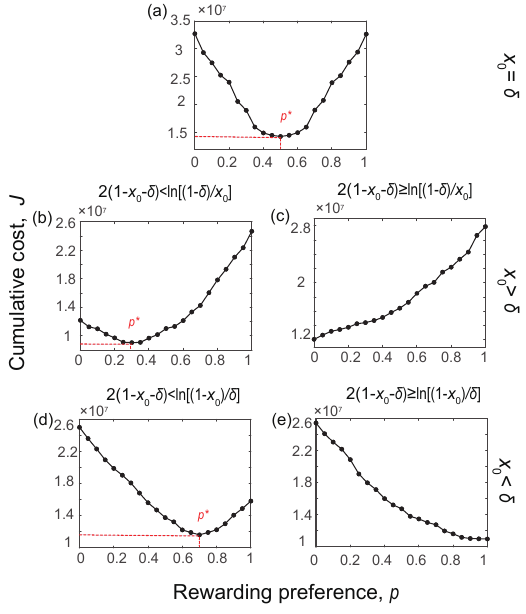}
				\caption{The required amount of cumulative cost to reach the expected proportion of cooperators in the population for the optimally combined incentive in dependence of the rewarding preference $p$ on scale-free networks. All results are derived from Monte Carlo simulations by averaging over 200 independent simulation runs on scale-free networks, which are obtained by using preferential-attachment model starting from $m_0=6$ where at every time step each new node is connected to $m=2$ existing nodes fulfilling the standard power-law distribution~\cite{Barab:1999}. Other parameters and descriptions are identical to Fig.~\ref{fig3}.
				}\label{fig4}
			\end{center}
		\end{figure}

		\subsubsection{$x_{0}>\delta$}
		From the second column of Fig.~\ref{fig2} (i.e., when $2(1-x_{0}-\delta)<\ln(\frac{1-\delta}{x_{0}})$), we find that the amount of cumulative cost $J^*_v(p)$ is monotonically decreasing over $[0, p^*)$, and monotonically increasing over $[p^*, 1]$. In addition,  $J^*_v(p)$ is minimal at  $p^\ast$ in the interval $[0, 1]$, indicating that the optimally combined incentive with $p^\ast$ is a lowest cost scheme for all $p\in [0, 1]$. Particularly, the $J_v^*(0)$ value is lower than the $J_v^*(1)$ value, which means that the cumulative cost value of optimal punishment is lower than the usage of optimal reward. On the other hand, from the third column of Fig.~\ref{fig2}, when $2(1-x_{0}-\delta)\leq \ln(\frac{1-\delta}{x_{0}})$, we see that $J^*_v(p)$ is monotonically increasing over  $[0,1]$,  then $J^*_v(0)$ is always less than  the values of $J^*_v(p)$ for all $p\in(0,1]$,  which means that  optimal punishment leads to the lowest cumulative cost. Therefore, we verify the theoretical results on complete graphs  obtained in Theorem \ref{thm4} and Lemma \ref{lem5}.
		
		\begin{figure}[!t]
			\begin{center}
				\includegraphics[width=3.5in]{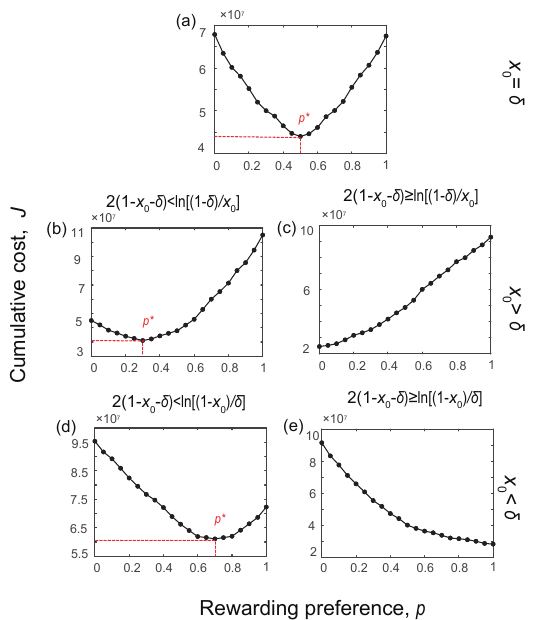}
				\caption{The required amount of cumulative cost to reach the expected proportion of cooperators in the population for the optimally combined incentive in dependence of the rewarding preference $p$ on small-world networks. All results are derived from Monte Carlo simulations by averaging over 200 independent simulation runs on small-world networks, which are obtained by using Watts-Strogatz model with $p=0.1$ rewiring parameter ~\cite{Watts:1998}. Other parameters and descriptions are identical to Fig.~\ref{fig3}.
				}\label{fig5}
			\end{center}
		\end{figure}

		\subsubsection{$x_{0}<\delta$}
		From the forth column of Fig.~\ref{fig2}, we see that when
		$2(1-x_{0}-\delta)<\ln(\frac{1-x_{0}}{\delta})$, the amount of cumulative cost $J^*_v(p)$ is monotonically decreasing over  $[0, p^*)$, while increasing over $[ p^*, 1]$. In addition, $J^*_v(p)$ is minimal at $p^\ast$ over $[0, 1]$. Specially, the $J^*_v(1)$ value is lower than the $J^*_v(0)$ value, which implies that the cumulative cost value of optimal reward is lower than the usage of optimal punishment.  On the contrary, from the fifth column of Fig.~\ref{fig2}, $J^*_v(p)$ is monotonically decreasing over $[0, 1]$,  then $J^*_v(1)$ is lower than the values of $J^*_v(p)$ for all $p\in[0,1)$, that is, optimal reward is cheaper incentive. Therefore, we verify the theoretical results on complete graphs obtained in Theorem \ref{thm5} and Lemma \ref{lem7}.
		
		Lastly, from the first and second rows of Fig.~\ref{fig2},
		we see that for any given scenario above-mentioned, our simulations results coincide with numerical calculations results, both of which support the analytical results on complete graphs as stated in Theorems \ref{thm3}-\ref{thm5}.

		\begin{figure}[!t]
			\begin{center}
				\includegraphics[width=3.5in]{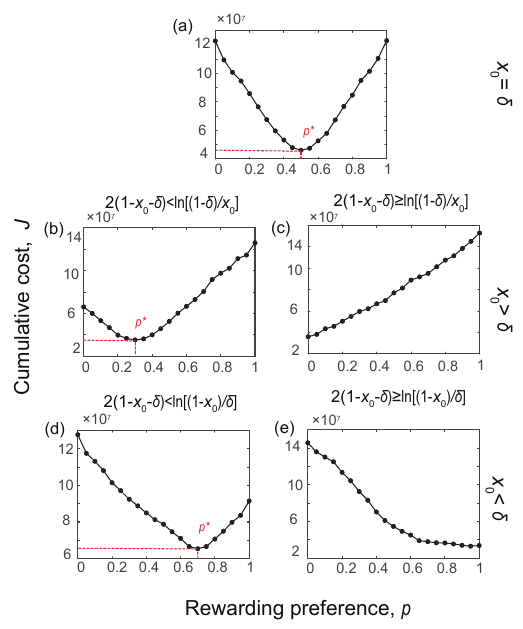}
				\caption{The required amount of cumulative cost to reach the expected proportion of cooperators in the population for the optimally combined incentive in dependence of the rewarding preference $p$ on random networks. All results are derived from Monte Carlo simulations by averaging over 200 independent simulation runs on random networks, which are obtained by using Erd\H{o}s-R\'{e}nyi random graph model~\cite{Erd:1959}. Other parameters and descriptions are identical to Fig.~\ref{fig3}.
				}\label{fig6}
			\end{center}
		\end{figure}

		\subsection{Alternative networks}
		As we noted, the applied network topologies could be a decisive factor for how the system evolves, therefore it is essential to explore whether our theoretical results are valid on alternative complex networks. Here, we have first considered regular networks, and accordingly, the numerical results are illustrated in Fig.~\ref{fig3}. Similar to complete graphs, we find that when $x_{0}=\delta$, the cumulative cost value of the optimally combined incentive is minimal at $p^\ast=0.5$; when $x_{0}>\delta$, if $2(1-x_{0}-\delta)<\ln(\frac{1-\delta}{x_{0}})$, the  optimally combined incentive with lowest cost is the one at $p=p^\ast$, otherwise the  optimally combined incentive with lowest cost is the one at $p=0$, that is, optimal  punishment; when $x_{0}<\delta$, if  $2(1-x_{0}-\delta)<\ln(\frac{1-x_{0}}{\delta})$,  the  optimally combined incentive with lowest cost is the one at $p=p^\ast$, otherwise the  optimally combined incentive with lowest cost is the one at $p=1$, that is, optimal reward.  Besides, from the first and second rows of Fig.~\ref{fig3}, our numerical calculations and simulations results are in line with the analytical results on regular networks.
		
		Furthermore, we have also tested other networks by performing simulations, including scale-free, small-world, and random networks, as shown in Figs.~\ref{fig4}-\ref{fig6}. From these figures, our simulation results reveal that our theoretical results obtained in Theorems \ref{thm3}-\ref{thm5} remain valid not merely for complete graphs  and regular networks, but also for other types of networks, including scale-free, small-world, and random networks.

		\section{Discussion and Conclusion}
		
		In this work, we study a combined incentive policy for promoting cooperation. We first derive dynamical equations  for the evolution of cooperation level by replicator equation on complete  graphs  and by employing the pair approximation approach in the weak selection limit on regular networks, respectively. For each given structured network, there always exist  two equilibria, which correspond to respectively the full defection and full cooperation states. By analyzing the stability of equilibria, we derive the minimal incentive level required for the evolution of cooperation. Subsequently, we establish an index function for quantifying the total executing cost, and accordingly obtain the optimally combined incentive protocols, which are identical and time-invariant both for two types of networks. Furthermore, we prove that when optimal
		reward and punishment are used alone, the relation of them highly depends on the initial cooperation level, that is, applying the optimal punishing scheme requires a lower cumulative cost for the incentive-providing institution when the initial cooperation level is relatively high; otherwise, applying the rewarding scheme is cheaper.  Nonetheless, when the optimally combined incentive is considered, we always yield the theoretical conditions, under which the cumulative cost induced by the optimally combined incentive is lowest. Our results shed some light on how to use the incentives for driving the game system to reach the desirable state more effectively from a cost-efficient way.

		\section*{appendix}
		In this section, we provide some proof details of the results shown in the main text.
		\subsection{Derivations of Equation \eqref{eq9}}\label{APPENDIXA}
		We introduce some notations. Let $x_{i}$ denote the proportion of $i$-agents, let $x_{ij}$ denote the proportion of $ij$-pairs, and finally let $x_{i|j}$  denote the conditional probability of finding an $i$-agent given that the neighboring node is occupied by a $j$-agent, where $i, j \in \{C, D\}$. By using these notations, we have that $x_{C}+x_{D}=1$, $x_{ij}=x_{i|j}x_{j}$, and $x_{CD}=x_{DC}$.
		
		Besides, it can be checked that $x_{C}$ and $x_{C|C}$ 
		can be used to describe the evolutionary dynamics of a dynamical system because $x_{D}=1-x_{C}$, $x_{D|C}= 1-x_{C|C}$, $ x_{C|D}=\frac{x_{CD}}{x_{D}}=\frac{x_{C}(1-x_{C|C})}{1-x_{C}}$, $x_{D|D}=1-x_{C|D}=\frac{1-2x_{C}+x_{C}x_{C|C}}{1-x_{C}}$, $x_{DC}= x_{CD}= x_{C}x_{D|C}= x_{C}(1-x_{C|C})$, and $x_{DD}=x_{D}x_{D|D}=1-2x_{C}+x_{C}x_{C|C}$.

			We consider the combined incentive into the  prisoner's dilemma game. For pairwise-comparison updating, we randomly choose two neighboring agents who adopts unequal strategies. Accordingly, if the selected $D$-agent updates its strategy and imitates a $C$-neighbor successfully, then the $n_{C}$ number of cooperators increases by one and the $n_{CC}$ number of $CC$-pairs increases by $1+(k-1)q_{C\mid D}$. This means that $x_{C}$ increases by $\frac{1}{n}$ and $x_{CC}$ increases by $\frac{1+(k-1)q_{C\mid D}}{kn/2}$.
			Thus, the transition probability $T^{+}(n_{C})$ that $n_{C}$ increases by one for $\omega \rightarrow 0$ is
			\begin{equation}\label{eq70}
				\begin{aligned}
					T^{+}(n_{C})&=x_{D}x_{C|D}W_{D\rightarrow C}\\
					&=x_{C}(1-x_{C|C})(\frac{1}{2}+\omega\frac{\Pi_{C}-\Pi_{D}}{4})+O(\omega^{2}).
				\end{aligned}
			\end{equation}
			
			Similarly, if the selected $C$-agent imitates a $D$-neighbor successfully, then $n_{C}$ decreases by one and $n_{CC}$ decreases by $(k-1)x_{C\mid C}$, indicating $x_{C}$ decreases by $\frac{1}{n}$ and $x_{CC}$ decreases by $\frac{(k-1)x_{C|C}}{kn/2}$.  Therefore, the transition probability $T^{-}(n_{C})$ that $n_{C}$ decreases by one for $\omega \rightarrow 0$ is
			\begin{equation}\label{eq71}
				\begin{aligned}
					T^{-}(n_{C})&=x_{C}x_{D|C}W_{C\rightarrow D}\\
					&=x_{C}(1-x_{C|C})(\frac{1}{2}+\omega\frac{\Pi_{D}-\Pi_{C}}{4})+O(\omega^{2}).
				\end{aligned}
			\end{equation}
			
			\begin{assumption}\label{assumption4}
				We assume that one replacement event occurs in each unit of time $t$.
			\end{assumption}
			
			Under Assumption \ref{assumption4}, one can derive the derivative of $x_{C}$ over  $t$ as
			\begin{equation}\label{eq72}
				\begin{aligned}
					\frac{\textrm{d}x_{C}(t)}{\textrm{d}t}&=\frac{\frac{1}{n}T^{+}(n_{C})-\frac{1}{n}T^{-}(n_{C})}{\frac{1}{n}}=T^{+}(n_{C})-T^{-}(n_{C})\\
					&=\omega\Psi(x_{C}, x_{C\mid C}, \omega),
				\end{aligned}
			\end{equation}
			where
			\begin{equation}\label{eq73}
				\begin{aligned}
					\Psi(x_{C}, x_{C\mid C}, \omega)&=\frac{x_{C}(1-x_{C|C})}{2} \Big[(k-1)b\frac{x_{C|C}-x_{C}}{1-x_{C}}\\
					&\quad +k(u-c)-b\Big]+O(\omega).
				\end{aligned}
			\end{equation}
			
			Besides, the time derivative of $x_{CC}$ is
			\begin{equation}\label{eq74}
				\begin{aligned}
					\frac{\textrm{d}x_{CC}(t)}{\textrm{d}t}&=\frac{\frac{1+(k-1)q_{C\mid D}}{kn/2}T^{+}(n_{C})- \frac{(k-1)x_{C|C}}{kn/2}T^{-}(n_{C})}{\frac{1}{n}}\\
					&=\frac{x_{C}(1-x_{C|C})}{k}\Big[1-(k-1)\frac{x_{C|C}-x_{C}}{1-x_{C}}\Big]+O(\omega).
				\end{aligned}
			\end{equation}
			
			Since $x_{C\mid C}=\frac{x_{CC}}{x_{C}}$, one obtains the derivative of $x_{C\mid C}$ with respect to  $t$ as
			\begin{equation}\label{eq75}
				\begin{aligned}
					\frac{\textrm{d}x_{C\mid C}(t)}{\textrm{d}t}=\Phi(x_{C}, x_{C\mid C}, \omega),
				\end{aligned}
			\end{equation}
			where
			\begin{equation}\label{eq76}
				\begin{aligned}
					\Phi(x_{C}, x_{C\mid C}, \omega)= \frac{1-x_{C|C}}{k}\Big[1-(k-1)\frac{x_{C|C}-x_{C}}{1-x_{C}}\Big]+O(\omega).
				\end{aligned}
			\end{equation}
			
			Combining \eqref{eq72} and \eqref{eq75},  the dynamical equation is described by
			\begin{equation}\label{eq77}
				\begin{aligned}
					\left\{\begin{array}{lc}
						\frac{\textrm{d}x_{C}(t)}{\textrm{d}t}=\omega \Psi(x_{C}, x_{C\mid C}, \omega),\\
						\frac{\textrm{d}x_{C\mid C}(t)}{\textrm{d}t}= \Phi(x_{C}, x_{C\mid C}, \omega).
					\end{array}\right.
				\end{aligned}
			\end{equation}
			System \eqref{eq77}  can be rewritten with a change in time scale as
			\begin{equation}\label{eq78}
				\left\{\begin{array}{lc}
					\frac{\textrm{d}x_{C}(\tau)}{\textrm{d}\tau}=\Psi(x_{C}, x_{C\mid C}, \omega),\\
					\omega\frac{\textrm{d}x_{C\mid C}(\tau)}{\textrm{d}\tau}=\Phi(x_{C}, x_{C\mid C}, \omega),
				\end{array}\right.
			\end{equation}
			where $\tau=\omega t$. We refer to the time scale given by $\tau$ as slow, whereas the
			time scale for $t$ is fast.
			Further, as long as $\omega\neq0$, the two systems above are equivalent
			and are referred to as singular perturbation when $0<\omega\ll1$. In \eqref{eq78},  letting $\omega\rightarrow0$, we obtain the
			system
			\begin{equation}\label{eq79}
				\left\{\begin{array}{lc}
					\frac{\textrm{d}x_{C}(\tau)}{\textrm{d}\tau}=\Psi(x_{C}, x_{C\mid C}, \omega),\\
					0=\Phi(x_{C}, x_{C\mid C}, \omega),
				\end{array}\right.
			\end{equation}
			which is called the reduced model.
			One thinks of the condition $\Phi(x_{C}, x_{C\mid C}, \omega)=0$ as determining a set on which the flow
			is given by $\frac{\textrm{d}x_{C}(\tau)}{\textrm{d}\tau}=\Psi(x_{C}, x_{C\mid C}, \omega)$.
			For $\omega=0$, the set $\mathcal{V}=\{(x_{C}, x_{C\mid C})\mid\Phi(x_{C}, x_{C\mid C}, 0)=0\}$ consists of two subsets
			\begin{equation}\label{eq80}
				\begin{aligned}
					\mathcal{M}_{0}^{0}=\big\{(x_{C}, x_{C\mid C})\big| x_{C\mid C}=1\big\}=\big\{(1, 1)\big\},
				\end{aligned}
			\end{equation}
			and
			\begin{equation}\label{eq81}
				\begin{aligned}
					\mathcal{M}_{0}^{1}=\Big\{(x_{C}, x_{C\mid C})\big| x_{C|C}=\frac{1}{k-1}+\frac{k-2}{k-1}x_{C}\Big\}.
				\end{aligned}
			\end{equation}
			It is noted that $\mathcal{M}_{0}^{0}$ is equal to $\{(1, 1)\}$ since $x_{CC}=x_{C}x_{C|C}$, that is, if $x_{C|C}=1$, then $x_{C}=1$.
			We prove this by contradiction. Suppose $x_{C}\neq1$, then there exists at least a $D$-agent in a population. Since the studied network is connected, at least a $C$-agent is linked to
			a $D$-agent, which leads to $x_{D|C}>0$. In addition, due to $x_{D|C}+ x_{C|C}=1$,  this implies that $x_{C|C}<1$,
			contradicting  $x_{C|C}=1$. Therefore, we verify the above conclusion.
			
			Based on Fenichel's Second Theorem\cite{Khalil96}, only $\mathcal{M}_{0}^{1}\subset \mathcal{V}$ is compact, with boundary and normally hyperbolic, and suppose $\Psi$ and $\Phi$  are smooth in $S$. Then for $\varepsilon>0$, which is sufficiently small, there exists a stable manifold $\mathcal{M}_{\varepsilon}^{1}$, that are $\mathcal{O}(\varepsilon)$ close and diffeomorphic to a stable manifold $\mathcal{M}_{0}^{1}$, and that are locally invariant under the flow of \eqref{eq81}. Hence, the manifold $\mathcal{M}_{0}^{1}$ is invariant manifold.
			In that case, the reduced model can characterize the dynamical equation of \eqref{eq78}  on $\mathcal{M}_{0}^{1}$, given by
			\begin{equation}\label{eq82}
				\begin{aligned}
					\frac{\textrm{d}x_{C}(\tau)}{\textrm{d}\tau} =\frac{k(k-2)}{2(k-1)}x_{C}(1-x_{C})(u-c).
				\end{aligned}
			\end{equation}
			Let $t=\frac{\tau}{\omega}$ and $x_{C}=x$, and the above equation is equivalent to
			\begin{equation}\label{eq83}
				\begin{aligned}
					\frac{\textrm{d}x(t)}{\textrm{d}t}=\frac{\omega k(k-2)}{2(k-1)}x(1-x)(u-c).
				\end{aligned}
			\end{equation}

			\subsection{Proof of Theorem~\ref{thm1}}\label{APPENDIXB}
			The Hamiltonian function $H_{I}(x, u, t)$ of \eqref{eq10} is defined as
			\begin{equation}\label{eq50}
				\begin{aligned}
					H_{I}(x, u, t)&=G_{I}(x, u, t)+\frac{\partial J_{I}^{\ast}}{\partial x}F_{I}(x, u,  t)\\
					&=\frac{1}{2}\big\{ n(n-1)u [p x+(1-p)(1-x)] \big\}^2\\
					&\quad +\frac{\partial J_{I}^{\ast}}{\partial x}x(1-x)(u-c),
				\end{aligned}
			\end{equation}
			where $J^\ast_I(x, t)$ is the optimal function of $x$ and $t$ for the optimally combined  incentive $u^{\ast}$, given by
			\begin{equation}\label{eq51}
				\begin{aligned}
					J_{I}^\ast(x, t)=\int^{t_{f}}_{t_0} G_{I}(x, u^{\ast}, t) \textrm{d}t.
				\end{aligned}
			\end{equation}
			Solving $\frac{\partial H_{I}}{\partial u}=0$, one can check that the optimally combined  incentive protocol  $u^{\ast}$ satisfies
			\begin{equation}\label{eq52}
				\begin{aligned}
					u^{\ast}=-\frac{x(1-x)}{\{n(n-1)[p x+(1-p)(1-x)]\}^{2}} \frac{\partial J_{I}^{\ast}}{\partial x}.
				\end{aligned}
			\end{equation}
			Traditionally, we shall solve the canonical equations of \eqref{eq50} to obtain the optimal incentive protocol.  However, the fact that the system \eqref{eq8}  is nonlinear makes it intractable to obtain the optimal incentive protocol by a direct calculation.  
			To this end, we employ  the approach of the Hamilton-Jacobi-Bellman equation to solve the optimal control problem \eqref{eq10}, and this equation can be written as
			\begin{equation}\label{eq53}
				\begin{aligned}
					-\frac{\partial J_{I}^\ast}{\partial t}&=H_{I}(x, u^{\ast}, t)=G_{I}(x, u^*, t)+\frac{\partial J_{I}^{\ast}}{\partial x}F_{I}(x, u^*,  t).
				\end{aligned}
			\end{equation}
			By substituting \eqref{eq52} into the above Hamilton-Jacobi-Bellman equation, one gets that
			\begin{equation}\label{eq54}
				\begin{aligned}
					-\frac{\partial J_{I}^{\ast}}{\partial t}&=\frac{1}{2}\Big\{ n(n-1)u^\ast \big[p x+(1-p)(1-x)\big] \Big\}^2\\
					&+\frac{\partial J_{I}^{\ast}}{\partial x}x(1-x)(u^{\ast}-c).
				\end{aligned}
			\end{equation}
			Under Assumption \ref{assumption2}, one knows that the optimal function $J_{I}^{\ast}(x, t)$ is independent of $t$, and then \eqref{eq54} turns out to be 
			\begin{equation}\label{eq55}
				\begin{aligned}
					\frac{\partial J_{I}^\ast}{\partial t}=0.
				\end{aligned}
			\end{equation}
			Furthermore, it is derived that 
			\begin{equation}\label{eq56}
				\begin{aligned}
					\frac{\partial J_{I}^{\ast}}{\partial x} = 0 \;\; {\rm or} \;\;
					\frac{\partial J_{I}^{\ast}}{\partial x}=-\frac{2c\{ n(n-1)[px+(1-p)(1-x)]\}^2}{x(1-x)}.
				\end{aligned}
			\end{equation}
			In view of the fact that  $u>0$ and $x\in(0, 1)$, $\frac{\partial J_{I}^{\ast}}{\partial x}$ is always negative (i.e., $\frac{\partial J_{I}^{\ast}}{\partial x}<0$). Thus,  the latter of above equation \eqref{eq56} can be satisfied, and by substituting it into \eqref{eq52}, one gets
			the optimally combined incentive protocol $u^{\ast}$ as
			\begin{equation}\label{eq57}
				\begin{aligned}
					u^{\ast}=2c. 
				\end{aligned}
			\end{equation}
			Furthermore, the system equation \eqref{eq8} can be rewritten as
			\begin{equation}\label{eq58}
				\begin{aligned}
					\frac{\textrm{d}x(t)}{\textrm{d}t}=x(1-x)c,
				\end{aligned}
			\end{equation}
			and its solution is 
			\begin{equation}\label{eq59}
				\begin{aligned}
					x(t)=\frac{1}{1+\frac{1-x_{0}}{x_{0}} e^{-ct}}.
				\end{aligned}
			\end{equation}
			By checking \eqref{eq59}, one can see that the system needs infinitely long time to reach a full cooperation state from any initial  state $x_0\in (0, 1)$.  Instead, under 
			Assumption \ref{assumption3}, the terminal state $x(t_{f})$ is assumed to be $1-\delta$, where $\delta$ is the parameter that determines the  proportion of cooperators in the population at the terminal time
			$t_{f}$. Since $x(t)$ increases monotonically over time $t$, then $x(t_{f})> x_{0}$ (i.e., $x_0+\delta<1$). Furthermore, one obtains the terminal time $t_{f}$ as
			\begin{equation}\label{eq60}
				\begin{aligned}
					t_{f}=\frac{1}{c}\ln\Big[\frac{(1-x_0)(1-\delta)}{x_0 \delta}\Big].
				\end{aligned}
			\end{equation}
			Combining \eqref{eq57}, \eqref{eq59}  and \eqref{eq60},  the amount of cumulative  incentive required by $u^{\ast}$ for the system to reach the expected terminal state $x(t_{f})$ from the initial state $x_0$ is
			\begin{equation}\label{eq61}
				\begin{aligned}
					J_{I}^\ast&=\int^{t_{f}}_{t_0} G_{I}(x, u^{\ast}, t) \textrm{d}t\\
					&=\int^{t_{f}}_{t_0}\frac{1}{2}\Big\{ n(n-1)u^\ast \big[p x+(1-p)(1-x)\big] \Big\}^2\textrm{d}t\\
					&=\alpha_{I} p^2+2\beta_{I} p(1-p)+\gamma_{I}(1-p)^2\\
					&=\vartheta_I \Big[(2p-1)^2(x_{0}-1+\delta)+p^2\ln\Big(\frac{1-x_0}{\delta}\Big)\\
					&\quad +(1-p)^2\ln\Big(\frac{1-\delta}{x_0}\Big)\Big],
				\end{aligned}
			\end{equation}
			where $\vartheta_I=2n^2(n-1)^2c$, $\alpha_{I}=\int^{t_{f}}_{t_0}\frac{1}{2}[n(n-1)xu^{\ast}]^{2}\textrm{d}t=\vartheta_I[x_{0}-1+\delta+\ln(\frac{1-x_0}{\delta})],$
			$\beta_{I}=\int^{t_{f}}_{t_0}\frac{1}{2}[n(n-1)u^{\ast}]^{2}x(1-x)\textrm{d}t=\vartheta_I(1-x_0-\delta)$, and 
			$\gamma_{I}=\int^{t_{f}}_{t_0}\frac{1}{2}[n(n-1)(1-x)u^{\ast}]^{2}\textrm{d}t=\vartheta_I[x_{0}-1+\delta+\ln(\frac{1-\delta}{x_0})].$

			\subsection{Proof of Theorem~\ref{thm2}}\label{APPENDIXC}
			The Hamiltonian function $H_{\textrm{S}}(x, u, t)$ of  \eqref{eq11} is
			\begin{equation}\label{eq62}
				\begin{aligned}
					H_{\textrm{S}}(x, u, t)&=G_{S}(x, u, t)+\frac{\partial J_{S}^{\ast}}{\partial x}F_{S}(x, u,  t)\\
					&=\frac{1}{2} \big\{ nku[px+(1-p)(1-x)] \big\}^2\\
					&+\frac{\partial J_{S}^{\ast}}{\partial x}\frac{\omega k(k-2)}{2(k-1)}x(1-x)(u-c),
				\end{aligned}
			\end{equation}
			where $J_{\textrm{S}}^\ast$ is the optimal cost function of $x$ and $t$ for the optimally combined incentive, given by
			\begin{equation}\label{eq63}
				\begin{aligned}
					J_{\textrm{S}}^\ast&=\int^{t_{f}}_{t_0} G_{S}(x, u^*, t) \textrm{d}t.
				\end{aligned}
			\end{equation}
			Solving $\frac{\partial H_{\textrm{S}}}{\partial u}=0$, one can check that  the optimally combined incentive protocol $u^{\ast}$ satisfies
			\begin{equation}\label{eq64}
				\begin{aligned}
					u^{\ast} =-\frac{x(1-x)}{\big\{nk[px+(1-p)(1-x)]\big\}^2} \frac{\omega k(k-2)}{2(k-1)}\frac{\partial J_{\textrm{S}}^{\ast}}{\partial x}.
				\end{aligned}
			\end{equation}
			The Hamilton-Jacobi-Bellman equation can be written as
			\begin{equation}\label{eq65}
				\begin{aligned}
					-\frac{\partial J_{\textrm{S}}^\ast}{\partial t}=H_{\textrm{S}}(x, u^{\ast}, t).
				\end{aligned}
			\end{equation}
			Under Assumption \ref{assumption2}, one gets 
			\begin{equation}\label{eq66}
				\begin{aligned}
					\frac{\partial J_{\textrm{S}}^\ast}{\partial t}=0,
				\end{aligned}
			\end{equation}
			and accordingly, yields that
			\begin{equation}\label{eq67}
				\begin{aligned}
					\frac{\partial J_{\textrm{S}}^{\ast}}{\partial x} = 0 \;\; {\rm or} \;\;
					\frac{\partial J_{\textrm{S}}^{\ast}}{\partial x}=-\frac{2c\big\{nk[px+(1-p)(1-x)]\big\}^2}{ \frac{\omega k(k-2)}{2(k-1)}x(1-x)}.
				\end{aligned}
			\end{equation}
			Since $u>0$ and $x\in(0, 1)$,  $\frac{\partial J_{\textrm{S}}^{\ast}}{\partial x}<0$.  Thus,  the latter of above equation \eqref{eq67} can be satisfied, and by futher substituting it into \eqref{eq64},  one obtains the optimally combined incentive protocol $u^{\ast}$ and the system equation \eqref{eq9}, respectively given by  
			\begin{equation}\label{eq68}
				\begin{aligned}
					u^{\ast}=2c,
				\end{aligned}
			\end{equation}
			and 
			\begin{equation}\label{eq69}
				\begin{aligned}
					\frac{\textrm{d}x(t)}{\textrm{d}t}= \frac{\omega k(k-2)c}{2(k-1)}x(1-x).
				\end{aligned}
			\end{equation}
			For the system \eqref{eq69},  its solution is
			\begin{equation}\label{eq70}
				\begin{aligned}
					x(t)=\frac{1}{1+\frac{1-x_{0}}{x_{0}} e^{-\frac{\omega k(k-2)c}{2(k-1)}t}}.
				\end{aligned}
			\end{equation}
			Under Assumption \ref{assumption3} and when $x(t_{f})=1-\delta$, one gets the terminal time  $t_{f}$ and the amount of cumulative  cost for the optimal level $u^{\ast}$, which respectively are
			\begin{equation}\label{eq71}
				\begin{aligned}
					t_{f}=\frac{2(k-1)}{\omega k(k-2)c}\ln\Big[\frac{(1-x_0)(1-\delta)}{x_0 \delta}\Big],
				\end{aligned}
			\end{equation}
			and
			\begin{equation}
				\begin{aligned}
					&J_{S}^\ast=\int^{t_{f}}_{t_0} G_{S}(x, u^{\ast}, t) \textrm{d}t\\
					&=\int^{t_{f}}_{t_0}\frac{1}{2} \Big\{ nku^\ast \big[px+(1-p)(1-x)\big] \Big\}^2\textrm{d}t\\
					&=\alpha_{S} p^2+2\beta_{S} p(1-p)+\gamma_{S}(1-p)^2\\
					&=\vartheta_{S} \Big[(2p-1)^2(x_{0}-1+\delta)+p^2\ln\Big(\frac{1-x_0}{\delta}\Big)\\
					&+(1-p)^2\ln\Big(\frac{1-\delta}{x_0}\Big)\Big],
				\end{aligned}
			\end{equation}
			where $\vartheta_S=\frac{4n^2k(k-1)c}{\omega(k-2)}$, 
			$\alpha_S=\int^{t_{f}}_{t_0}\frac{1}{2}(knxu^{\ast})^{2}\textrm{d}t=\vartheta_S[x_{0}-1+\delta+\ln(\frac{1-x_0}{\delta})],$
			$\beta_S=\int^{t_{f}}_{t_0}\frac{1}{2}(nku^{\ast})^{2}x(1-x)\textrm{d}t=\vartheta_S(1-x_0-\delta),$ and
			$\gamma_S=\int^{t_{f}}_{t_0}\frac{1}{2}[kn(1-x)u^{\ast}]^{2}\textrm{d}t=\vartheta_S[x_{0}-1+\delta+\ln(\frac{1-\delta}{x_0})].$

			\begin{IEEEbiography}[{\includegraphics[width=1in,height=1.25in,clip,keepaspectratio]{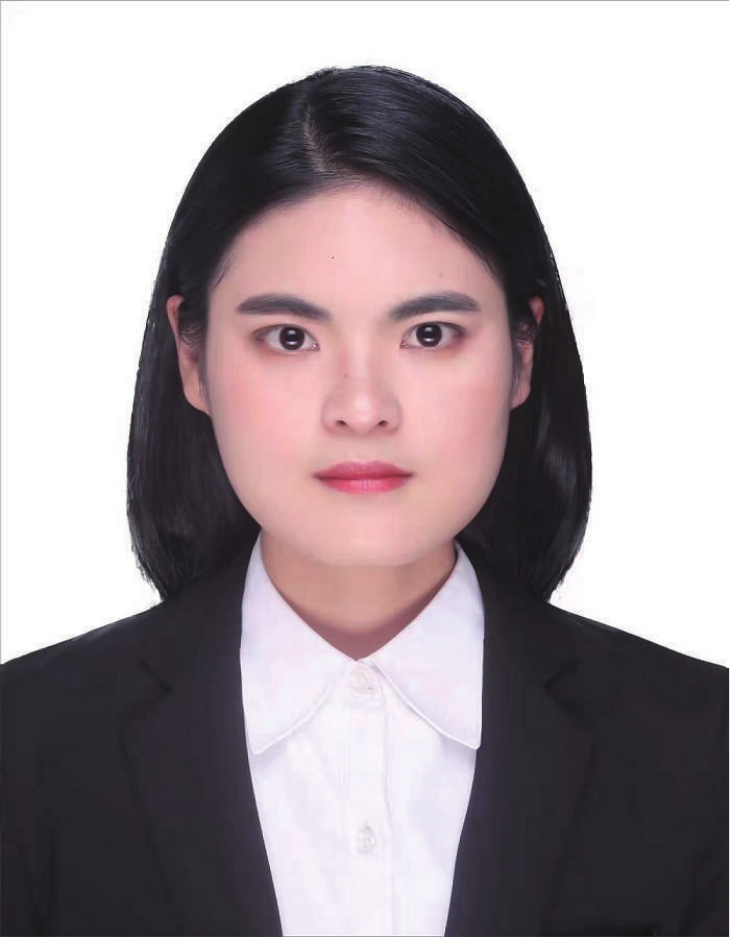}}]{Shengxian Wang} received the Ph.D. degree in mathematics from University of Electronic Science and Technology of China, Chengdu, China, in 2023. From December 2020 to December 2022, Dr. Wang was Guest Ph.D. in University of Groningen, Groningen, The Netherlands. She currently works with Anhui Normal University, Wuhu, China. Her research interests include evolutionary game dynamics, decision-making in game interactions, game-theoretical control, and collective intelligence. 
			\end{IEEEbiography}
			
			\begin{IEEEbiography}[{\includegraphics[width=1in,height=1.25in,clip,keepaspectratio]{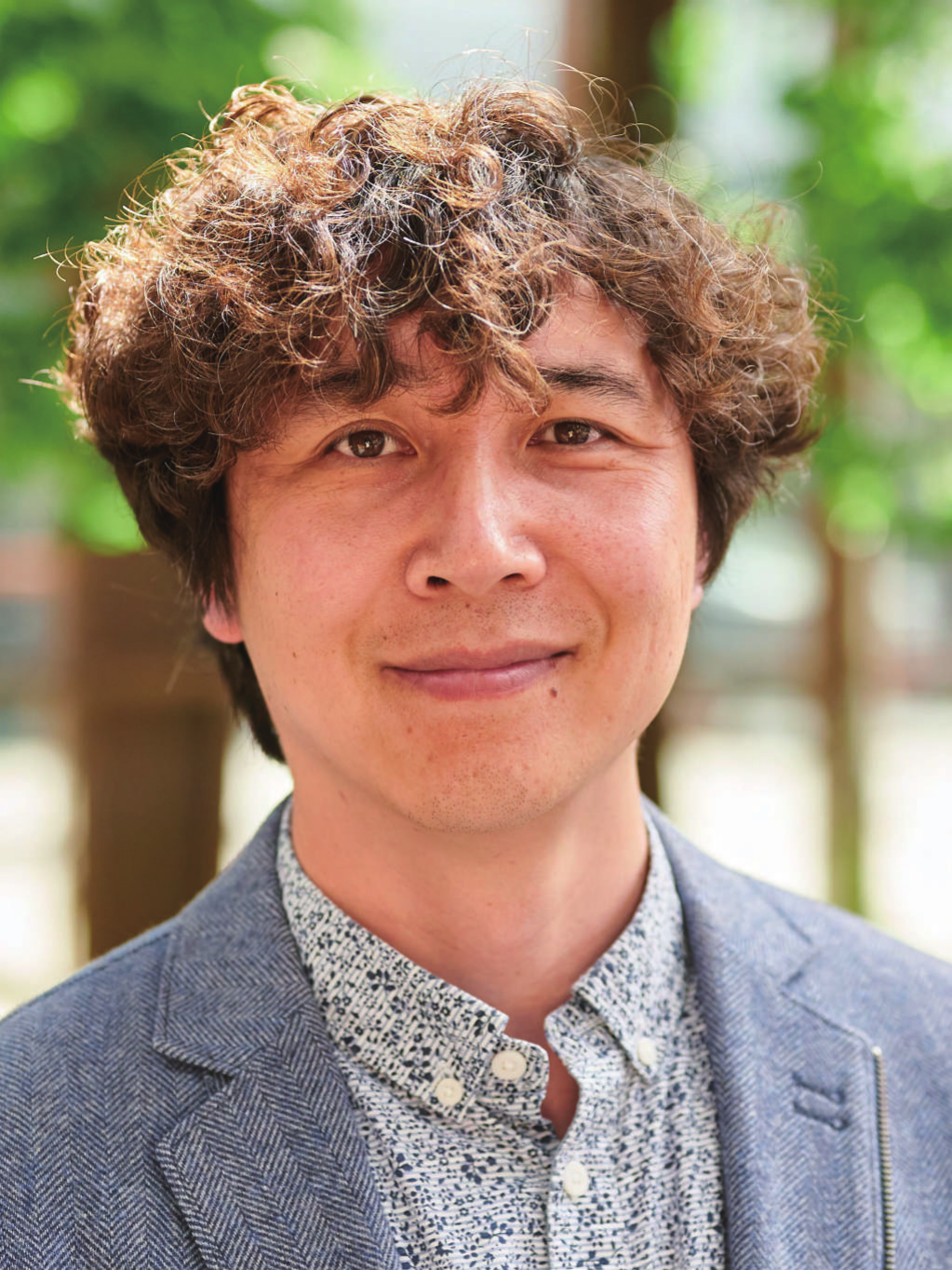}}]{Ming Cao} (Fellow, IEEE) has since 2016 been a professor of networks and robotics with the Engineering and Technology Institute (ENTEG) at the University of Groningen, the Netherlands, where he started as an assistant professor in 2008. Since 2022 he is the director of the Jantina Tammes School of Digital Society, Technology and AI at the same university. He received the Bachelor degree in 1999 and the Master degree in 2002 from Tsinghua University, China, and the Ph.D. degree in 2007 from Yale University, USA. From 2007 to 2008, he was a Research Associate at Princeton University, USA. He worked as a research intern in 2006 at the IBM T. J. Watson Research Center, USA. He is the 2017 and inaugural recipient of the Manfred Thoma medal from the International Federation of Automatic Control (IFAC) and the 2016 recipient of the European Control Award sponsored by the European Control Association (EUCA). He is an IEEE fellow. He is a Senior Editor for Systems and Control Letters, an Associate Editor for IEEE Transactions on Automatic Control, IEEE Transaction of Control of Network Systems and IEEE Robotics \& Automation Magazine, and was an associate editor for IEEE Transactions on Circuits and Systems and IEEE Circuits and Systems Magazine. He is a member of the IFAC Council and a vice chair of the IFAC Technical Committee on Large-Scale Complex Systems. His research interests include autonomous robots and multi-agent systems, complex networks and decision-making processes.
			\end{IEEEbiography}
			
			\begin{IEEEbiography}[{\includegraphics[width=1in,height=1.25in,clip,keepaspectratio]{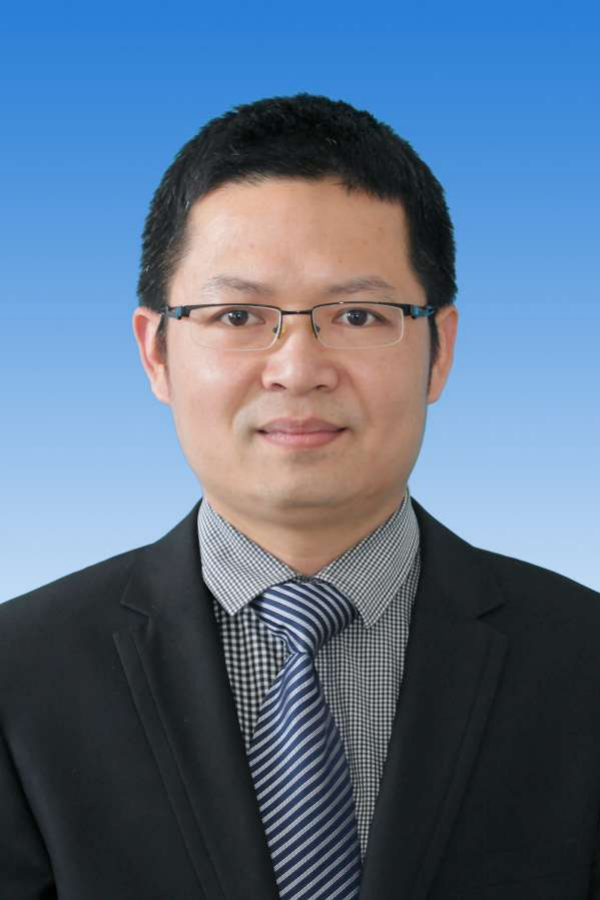}}]{Xiaojie Chen} is currently a professor at School of Mathematical Sciences in University of Electronic Science and Technology of China, China. He received the Bachelor degree in 2005 from National University of Defense Technology, China, and the Ph.D. degree in 2011 from Peking University, China. From September 2008 to September 2009, he was a visiting scholar in University of British Columbia, Canada. From February 2011 to January 2013, he was a postdoctoral research scholar at the International Institute for Applied Systems Analysis (IIASA), Austria. From February 2013 to January 2014, he was a research scholar at IIASA, Austria. He severs as the editorial board member for several international journals including Scientific Reports, PLoS ONE, and Frontiers in Physics. His main research interests include evolutionary game dynamics, decision-making in game interactions, game-theoretical control, and collective intelligence. He has published over $100$ journal papers.
			\end{IEEEbiography}


\section*{References}
			
			\begin{thebibliography}{00}
				
				
				\bibitem{Seeley: 1988} 
				T.~D.~Seeley, and P.~K.~Visscher, ``Assessing the benefits of cooperation in honeybee foraging: search costs, forage quality, and competitive ability,'' {\it Behav. Ecol. Sociobiol.}, vol. 22, pp. 229--237,  1988.
				
				
				
				\bibitem{Dugatkin: 2000} 
				L.~A.~Dugatkin, ``Cheating monkeys and citizen bees: the nature of cooperation in animals and humans,'' in {\it Harvard University Press}, 2000.
				
				
				\bibitem{Moon: 2020} 
				M.~J.~Moon, ``Fighting COVID‐19 with agility, transparency, and participation: wicked policy problems and new governance challenges,'' {\it Public Adm. Rev.}, vol. 80, no. 4, pp. 651--656,  2020.
				
				
				
				\bibitem{Momtazmanesh: 2020} 
				S.~Momtazmanesh, H.~D.~Ochs,  L.~Q.~Uddin, M.~Perc~{{et al.}}, ``All together to fight COVID-19,''~{\it Am. J. Trop. Med. Hyg.}, vol. 102, no. 6, pp. 1181--1183, 2020.
				
				
				
				\bibitem{Vincent:CUP} 
				T.~L.~Vincent, and J.~S.~Brown, ``Evolutionary Game Theory, Natural Selection, and Darwinian Dynamics,''  {\it Cambridge University Press}, 2005.
				
				
				
				\bibitem{Axelrod:BBNY} 
				R.~M.~Axelrod, ``The Evolution of Cooperation,''  {\it Basic Books: New York}, 2006.
				
				\bibitem{Rand:TCS} 
				D.~G.~Rand, and  M.~A.~Nowak, ``Human cooperation,''~{\it Trends Cogn. Sci.}, vol. 17, no. 8, pp. 413--425, 2013.
				
				
				
				\bibitem{Sunhb2022} 
				Q.~Su, A.~McAvoy, Y.~Mori, and J.~B.~Plotkin, ``Evolution of prosocial behaviours in multilayer populations,''~{\it Nat. Hum. Behav.}, vol. 6, no. 3, pp. 338--348, 2022.
				
				
				
				
				
				\bibitem{Hu: 2021} 
				Z.~Hu, X.~Li, J.~Wang, C.~Xia, Z.~Wang, and M.~Perc, ``Adaptive reputation promotes trust in social networks,''~{\it IEEE Trans. Netw. Sci. Eng.}, vol. 8, no. 4, pp. 3087--3098, 2021.
				
				
				\bibitem{RamaziTAC2017} 
				P.~Ramazi, and M.~Cao, ``Asynchronous decision-making dynamics under best-response update rule in finite heterogeneous populations,''~{\it IEEE Trans. Automat. Contr.}, vol. 63, no. 3, pp. 742--751, 2017.
				
				\bibitem{Sunpnas2022} 
				Q.~Su, B.~Allen, and J.~B.~Plotkin, ``Evolution of cooperation with asymmetric social interactions,''~{\it Proc. Natl. Acad. Sci. USA}, vol. 119, no. 1, pp. e2113468118, 2022.
				
				
				
				
				\bibitem{Smith: 1982} 
				J.~M.~Smith, ``Evolution and the Theory of Games,''  {\it Cambridge University Press}, 1982.
				
				\bibitem{Gintis: 2009} 
				H.~Gintis, ``Game Theory Evolving,'' {\it Princeton University Press}, 2009.
				
				\bibitem{Santos:2008} 
				F.~C.~Santos, M.~D.~Santos, and J.~M.~Pacheco, ``Social diversity promotes the emergence of cooperation in public goods games,''~{\it  Nature}, vol. 454, no. 7201, pp. 213--216, 2008.
				
				\bibitem{RiehlTAC2016} 
				J.~R.~Riehl, and M.~Cao, ``Towards optimal control of evolutionary games on networks,''~{\it IEEE Trans. Automat. Contr.}, vol. 62, no. 1, pp. 458--462, 2016.
				
				\bibitem{LiNC2020} 
				A.~Li, L.~Zhou, Q.~Su, S.~P.~Cornelius,  Y.~Y.~Liu,  L.~Wang, and S.~A.~Levin, ``Evolution of cooperation on temporal networks,''~{\it Nat. Commun.}, vol. 11, no. 1, pp. 2259, 2020.
				
				
				\bibitem{Zhang: 2018} 
				J.~Zhang, Y.~Zhu,  and Z.~Chen, ``Evolutionary game dynamics of multiagent systems on multiple community networks,''~{\it IEEE Trans. Syst. Man Cybern. Syst.}, vol. 50, no. 11, pp. 4513--4529, 2018.
				
				\bibitem{MadeoTAC2014} 
				D.~Madeo, and C.~Mocenni, ``Game interactions and dynamics on networked populations,''~{\it IEEE Trans. Automat. Contr.}, vol. 60, no. 7, pp. 1801--1810, 2014.
				
				
				
				\bibitem{GovaertTAC2020} 
				A.~Govaert, and M.~Cao, ``Zero-determinant strategies in repeated multiplayer social dilemmas with discounted payoffs,''~{\it IEEE Trans. Automat. Contr.}, vol. 66, no. 10, pp. 4575--4588, 2020.
				
				\bibitem{Chen2023TAC} 
				G.~Chen, and Y.~Yu, ``Convergence analysis and strategy control of evolutionary games with imitation rule on toroidal grid,''~{\it IEEE Trans. Automat. Contr.}, to be published, doi: 10.1109/TAC.2023.3291957.
				
				\bibitem{Nowak:1992} 
				M.~A.~Nowak, and R.~M.~May, ``Evolutionary games and spatial chaos,''~{\it Nature}, vol. 359, no. 6398, pp. 826--829, 1992.
				
				\bibitem{LiTEC2016} 
				J.~Li, C.~Zhang, Q.~Sun, Z.~Chen, and J.~Zhang, ``Changing the intensity of interaction based on individual behavior in the iterated prisoner’s dilemma game,''~{\it IEEE Trans. Evol. Comput.}, vol. 21, no. 4, pp. 506--517, 2016.
				
				\bibitem{HilbePNAS13} 
				C.~Hilbe, M.~A.~Nowak, and K.~Sigmund, ``Evolution of extortion in
				iterated prisoner’s dilemma games,''~{\it Proc. Nat. Acad. Sci. USA}, vol. 110, no. 17, pp. 6913--6918, 2013.
				
				
				
				
				
				
				\bibitem{Sasaki: 2012} 
				T.~Sasaki, {\AA}.~Br\"{a}nnstr\"{o}m, U.~Dieckmann, and K.~Sigmund, ``The take-it-or-leave-it option allows small penalties to overcome social dilemmas,''~{\it Proc. Natl. Acad. Sci. USA}, vol. 109, no. 4, pp. 1165--1169, 2012.
				
				
				\bibitem{Riehl1: 2018} 
				J.~R.~Riehl, P.~Ramazi, and  M.~Cao, ``Incentive-based control of asynchronous best-response dynamics on binary decision networks,''~{\it IEEE Trans. Control. Netw. Syst.}, vol. 6, no. 2, pp. 727--736, 2018.
				
				\bibitem{Vasconcelos: 2013} 
				V.~V.~Vasconcelos, F.~C.~Santos, and  J.~M.~Pacheco, ``A bottom-up institutional approach to cooperative governance of risky commons,''~{\it Nat. Clim. Change}, vol. 3, no. 9, pp. 797--801, 2013.
				
				
				\bibitem{Paarporn: 2018} 
				K.~Paarporn, and C.~Eksin, ``Incentive control in network anti-coordination games with binary types,'' in {\it 2018 52nd Asilomar Conference on Signals, Systems, and Computers, IEEE}, pp. 316--320, 2018.
				
				
				\bibitem{Ntemos: 2021} 
				K.~Ntemos,  G.~Pikramenos, and N.~Kalouptsidis, ``Dynamic information sharing and punishment strategies,''~{\it IEEE Trans. Automat. Contr.}, vol. 67, no. 4, pp. 1837--1852, 2022.
				
				
				\bibitem{Zhutac2022} 
				Y.~Zhu, C.~Xia, and Z.~Chen, ``Nash equilibrium in iterated multiplayer games under asynchronous best-response dynamics,''~{\it IEEE Trans. Automat. Contr.}, vol. 69, no. 9, pp.  5798--5805, 2023.
				
				\bibitem{TanTAC2016} 
				S.~Tan, Y.~Wang, and J.~L\"{u}, ``Analysis and control of networked game dynamics via a microscopic deterministic approach,''~{\it IEEE Trans. Automat. Contr.}, vol. 61, no. 12, pp. 4118--4124, 2016.
				
				
				\bibitem{MorimotoTAC2015} 
				T.~Morimoto, T.~Kanazawa, and T.~Ushio, ``Subsidy-based control of heterogeneous multiagent systems modeled by replicator dynamics,''~{\it IEEE Trans. Automat. Contr.}, vol. 61, no. 10, pp. 3158--3163, 2015.
				
				
				\bibitem{Fang:2019PRSA} 
				Y.~Fang, T.~P.~Benko, M.~Perc, H.~Xu, and Q.~Tan, ``Synergistic third-party rewarding and punishment in the public goods game,''~{\it  Proc. R. Soc. A}, vol. 475, no. 2227, pp. 20190349, 2019.
				
				
				
				\bibitem{Chen: 2015} 
				X.~Chen, T.~Sasaki, {\AA}.~Br\"{a}nnstr\"{o}m, and  U.~Dieckmann, ``First carrot, then stick: how the adaptive hybridization of incentives promotes cooperation,''~{\it J. R. Soc. Interface}, vol. 12, no. 102, pp. 20140935, 2015.
				
				
				
				
				\bibitem{Alventosa: 2021} 
				A.~Alventosa, A.~Antonioni, and P.~Hern\,{a}ndez, ``Pool punishment in public goods games: How do sanctioners incentives affect us?''~{\it J. Econ. Behav. Organ.}, vol. 185, pp. 513--537, 2021.
				
				
				\bibitem{Smith: 1977} 
				F.~J.~Smith, ``Work attitudes as predictors of attendance on a specific day,''~{\it J. Appl. Psychol.}, vol. 62, no. 1, pp. 16--19, 1977.
				
				\bibitem{Allen: 1981} 
				S.~G.~Allen, ``An empirical model of work attendance,''~{\it Rev. Econ. Stat.}, vol. 63, no. 1, pp. 77--87, 1981.
				
				\bibitem{Willard: 2020} 
				A.~K.~Willard, A.~Baimel, H.~Turpin, J.~Jong, and H.~Whitehouse, ``Rewarding the good and punishing the bad: The role of karma and afterlife beliefs in shaping moral norms,''~{\it Evol. Hum. Behav.}, vol. 41, no. 5, pp. 385--396, 2020.
				
				
				\bibitem{Duong: 2021} 
				M.~H.~Duong,  and T.~A.~Han, ``Cost efficiency of institutional incentives for promoting cooperation in finite populations,''~{\it Proc. R. Soc. A}, vol. 477, no. 2254, pp. 20210568, 2021.
				
				\bibitem{Wang:2019} 
				S.~Wang, X.~Chen, and A.~Szolnoki, ``Exploring optimal institutional incentives for public cooperation,''~{\it Commun. Nonlinear Sci. Numer. Simul.}, vol. 79, pp. 104914, 2019.
				
				\bibitem{Wang2022JRCS} 
				S.~Wang, X.~Chen, Z.~Xiao, A.~Szolnoki, and V.~V.~Vasconcelos, ``Optimization of institutional incentives for promoting cooperation in structured populations,''~{\it J. R. Soc. Interface}, vol. 20, no. 199, pp. 20220653, 2023.
				
				
				\bibitem{Evans: 2005} 
				L.~C.~Evans, ``An Introduction to Mathematical Optimal Control Theory,''  {\it University of California Press}, 2005.
				
				
				\bibitem{Geering: 2007} 
				H.~P.~Geering, ``Optimal Control with Engineering Applications,''  {\it Springer}, 2007.
				
				\bibitem{Schuster: 1983} 
				P.~Schuster, and K.~Sigmund, ``Replicator dynamics,''~{\it Journal of Theoretical Biology}, vol. 100, no. 3, pp. 533--538, 1983.
				
				\bibitem{Hofbauer: 1998} 
				J.~Hofbauer, and K.~Sigmund, ``Evolutionary Games and Population Dynamics,''  {\it Cambridge University Press}, 1998.
				
				\bibitem{Szabo: 1998} 
				G.~Szab\'{o}, and C.~T\H{o}ke, ``Evolutionary prisoner's dilemma game on a square lattice,''~{\it Phys. Rev. E}, vol. 58, no. 1, pp. 69, 1998.
				
				\bibitem{Ohtsuki: 2006} 
				H.~Ohtsuki, and M.~A.~Nowak, ``The replicator equation on graphs,''~{\it J. Theor. Biol.}, vol. 243, no. 1, pp. 86--97, 2006.
				
				
				\bibitem{Ohtsuk:2006Nature} 
				H.~Ohtsuki, C.~Hauert,  E.~Lieberman,  and M.~A. Nowak, ``A simple rule for the evolution of cooperation on graphs and social networks,''~{\it Nature}, vol. 441, no. 7092, pp. 502--505, 2006.
				
				
				\bibitem{Barab:1999} 
				A.~L.~Barab\'{a}si, and R.~Albert, ``Emergence of scaling in random networks,''~{\it Science}, vol. 286, no. 5439, pp. 509--512, 1999.
				
				\bibitem{Erd:1959} 
				P.~Erd\H{o}s, and A.~R\'{e}nyi, ``On random graphs I,''~{\it Publ. Math. Debrecen}, vol. 6, no. 18, pp. 290--297, 1959.
				
				
				
				\bibitem{Watts:1998} 
				D.~J.~Watts, and  S.~H.~Strogatz, ``Collective dynamics of `small-world' networks,''~{\it Nature}, vol. 393, no. 6684, pp. 440--442, 1998.
				
				
				
				
				\bibitem{Khalil96} 
				H.~K.~Khalil, ``Nonlinear Systems,''  {\it Prentice Hall, NJ}, 1996.
				
				
				
				
				
				
				
				
				
				
				
				
				
				
				
				
				
				
				
				
				
				
				
				
				
				
				
				
				
				
				
				
				
				
				
			\end{thebibliography}
\end{document}